\documentclass[12pt]{amsart}

\usepackage{amsmath,amssymb,amsthm,xypic,epsfig }
\usepackage{enumerate}
\usepackage{color}
\usepackage{bbm}
\usepackage{appendix}
\usepackage[margin=30truemm]{geometry}
\usepackage[colorinlistoftodos]{todonotes}
\usepackage{version}
\usepackage{tikz}
\usetikzlibrary{shapes.geometric, shapes.misc}

\DeclareFontFamily{OT1}{rsfs}{}
\DeclareFontShape{OT1}{rsfs}{n}{it}{<-> rsfs10}{}
\DeclareMathAlphabet{\curly}{OT1}{rsfs}{n}{it}

\newtheorem{Thm}{Theorem}[section]
\newtheorem{lem}[Thm]{Lemma}

\newtheorem{prop}[Thm]{Proposition}
\newtheorem{conj}[Thm]{Conjecture}

\newtheorem{Ques}{Question}

\theoremstyle{remark}
\newtheorem{Rem}[Thm]{Remark}
\newtheorem{ex}[Thm]{Example}
\theoremstyle{definition}
\newtheorem{defn}[Thm]{Definition}

\newtheorem{Setup}{Setup}
\newtheorem*{ack}{Acknowledgments}

\newcommand{\Spec}{\mathop{\mathrm{Spec}}\nolimits}

\newcommand{\Proj}{\mathop{\mathrm{Proj}}\nolimits}

\newcommand{\ord}{\mathop{\mathrm{ord}}\nolimits}
\newcommand{\gr}{\mathop{\mathrm{gr}}\nolimits}

\newcommand{\mult}{\mathop{\mathrm{mult}}\nolimits}

\newcommand{\lcth}{\mathop{\mathrm{lct}}\nolimits}

\def\mult{\operatorname{mult}}

\def\dim{\operatorname{dim}}

\def\DH{\operatorname{DH}}

\def\vol{\operatorname{vol}}
\def\deg{\operatorname{deg}}

\def\k{\mathbbm{k}}
\def\R{\mathbb{R}}

\def\C{\mathbb{C}}
\def\G{\mathbb{G}}
\def\Z{\mathbb{Z}}
\def\O{\mathcal{O}}
\def\P{\mathbb{P}}
\def\Q{\mathbb{Q}}

\def\A{\mathbb{A}}

\def\X{\mathcal{X}}
\def\Y{\mathcal{Y}}

\def\F{\mathcal{F}}

\def\W{\mathbb{W}}

\excludeversion{memo}

\begin{document}

\title[On Fano fibrations after Sun-Zhang]{On Sun-Zhang's theory of 
Fano fibrations - \\ 
weighted volumes, moduli and bubbling 
Fano fibrations}
\dedicatory{In honor of Prof. Caucher Birkar}
\author{Yuji Odaka}

\maketitle
\thispagestyle{empty}

\date{\today}

%\setcounter{tocdepth}{4} 
%\noindent
\tableofcontents

%%%abstract%%%
\begin{abstract}
We revisit the recent theory of Sun-Zhang on general Fano fibration 
which emerged from the study of non-compact K\"ahler-Ricci soliton metrics  
\cite{SZ}, primarily from an algebro-geometric perspective. 

In addition to reviewing the existing framework, we present new results, conjectures, and remarks. These include methods for computing weighted volumes via (restricted) volumes, Laplace transforms, and incomplete $\Gamma$-functions, and a 
conjectural 
algebro-geometric construction (``bubbling")  
of Fano fibration with asymptotically conical base 
from degenerating Fano fibration. 
\end{abstract}
%%%%%%%%%%

\section{Introduction}\label{sec:intro}

We work on algebraic schemes
over an algebraically closed field $\k$, mainly {\it ($\Q$-)Fano fibrations} in a  general sense, i.e. a projective surjective 
morphism $X\xrightarrow{\pi}Y$ where 
$X, Y$ is normal $\k$-varieties, 
$X$ is $\Q$-Gorenstein log terminal, $\pi_*\mathcal{O}_X=\mathcal{O}_Y$, 
$-K_X$ is $\pi$-ample. 
Note that 
classically, ``fibration" often refers to 
those with positive relative dimensions (sometimes with even 
flatness assumption), but here our 
$\pi$ does {\it not} need to be flat nor {\it equidimensional}, 
and can be even birational (e.g., blow downs). 
For this reason, this is also called {\it Fano 
contraction} and is also equivalent to 
{\it extremal contraction} in Mori's theory. 
In the case when the relative dimension 
is positive, 
it is also called {\it Mori fibration} or 
{\it Mori fiber space}, especially when 
the relative Picard number is $1$. 

When we work on some relations with 
complex geometry, we often suppose $\k=\C$ and whenever we 
(sometimes implicitly) use 
the minimal model program, we assume its characteristic is $0$ 
but otherwise the statements hold true for general $\k$. The perspective of \cite{SZ} 
(and hence this paper) 
is primarily 
based on the recent developments in 
K\"ahler geometry and related 
K-stability, as well as birational geometry, 
but provides a new interesting connection. To understand 
that, 
let us look further back at earlier stories. 

\vspace{2mm}

Since the groundbreaking work of S.~Mori \cite{Mori}, 
and subsequent developments worked out by 
Y.~Kawamata, M.~Reid, X.~Benveniste, 
V.~Shokurov, 
J.~Koll\'ar and others (cf., e.g., 
\cite[Chapter 3]{KollarMori} and 
the well-known references therein), 
it has been understood 
that these $\Q$-Fano fibrations in the above sense 
are the basic important structures to understand 
right process of the minimal model program since the Italian school. Indeed, the notion 
includes Castelnuovo's 
$(-1)$-curves contraction, ruled surfaces, Del Pezzo surfaces, Fano manifolds among others 
(e.g., later found flipping contractions). 
The author believes that now 
it should 
go without saying that, as many experts know, 
the theory developed to one of fundamental tools of 
algebraic geometry due to many contributors. 
The major breakthrough in higher dimensional case was 
done in Birkar-Cascini-Hacon-Mckernan's work 
\cite{BCHM}. 

Around the same time as \cite{Mori}, R.Hamilton \cite{Hamilton} 
(with the same publication years!) introduced in differential geometry 
a geometric flow of Riemannian metrics, 
the so-called {\it Ricci flow}, to apply to classification problems of differentiable manifolds.  
Later Perelman \cite{Perelman1, Perelman2} 
developed the idea of {\it Ricci flow with surgery} 
and solved 
the Poincar\'e  conjecture 
as well as 
the geometrization conjecture of Thurston 
(\cite{Thurston}\footnote{again the same publication years as \cite{Mori, Hamilton}!}), 
in geometry of (real) $3$-dimensional manifolds.  

\vspace{2mm}

After the works in complex geometry of H.Tsuji (cf., e.g., 
\cite{Tsuji}) as developped by Cascini-LaNave 
(\cite{CL}),  
Song-Tian (\cite{ST}) among 
others, 
now we understand that 
K\"ahler version i.e., the so-called K\"ahler-Ricci flow is compatible 
with the so-called minimal model program with scaling \cite{BCHM} and 
gives a bridge 
between these two (originally independent) studies. 
The flow (usually) ``stops" at finite time, where some singularities 
evolve, and it is observed that 
(cf., e.g., \cite{EMT11, Nab10, Bam20, CCD2, CHM25}) 
they rescale up to give so-called 
{\it complete 
gradient shrinking K\"ahler-Ricci solitons} ({\it shrinkers}, in  
short) which are self-similar solutions to the 
K\"ahler-Ricci-flow. 

Very recently, S.~Sun and J.~Zhang \cite{SZ} remarkably 
proved that such shrinkers are, at least in smooth case, 
always quasi-projective and even 
admits the {\it Fano fibration} structure, 
precisely in the sense described at the 
beginning of this introduction. 
Their work precisely connected the 
theory back to the original finding of Mori \cite{Mori}. 
Their proof cleverly uses the variation of K\"ahler quotients 
along the perturbations of soliton vector fields, consider their  
birational behaviour which somewhat parallels variation of GIT quotient 
(VGIT), and then apply a deep boundedness result of C.Birkar (\cite{BAB}). 
Note that Birkar has also already developed 
various boundedness type results for Fano fibrations and control of 
singularities of the base (cf., \cite{Birkar.before, BAB, 
Birkar.bdd, Birkar.sing, BC}).

After proving the above mentioned structure 
theorem, 
Sun-Zhang \cite{SZ} 
further introduced an invariant of 
Fano fibration (germ), as 
a notable enhancement of the bridge, which 
they call {\it weighted volume} 
and denote as $\W(-)$. 
This theory conjecturally leads to 
the generalized theory of K-stability, (K-)moduli space for Fano fibrations 
in the context of canonical K\"ahler metrics and K-stability 
theories. So, it 
extends the K-stability theory of 
Fano varieties, for instance. (For the case when $\pi={\rm id}$, i.e., the case of 
log terminal cones, 
see e.g., \cite{LLX, Od24a} and references 
therein). 

This note means to be a supplement to their notable work, especially on
algebro-geometric sides including moduli discussion and 
bubbling along a degeneration of Fano fibrations. 
The bulk of the paper is of expository nature, but
we also include various new propositions 
and small ideas. To make the exposition 
relatively self-contained, many parts are devoted to review of the theory of Sun-Zhang 
from purely algebro-geometric side. 

%%%%%%%%%

\subsection{Formulae 
of weighted volumes via (restricted) volumes}

Now we recall the definition 
of weighted volume $\W(-)$ 
and discuss basic properties. 
This integrates the earlier analytic definitions by 
\cite[\S 7]{CDS} (cf., also \cite{TZ} for when $Y$ is a point). 
We start with re-writing their definition in an equivalent way 
for future computation purpose. 
For that, we first prepare a non-compact variant as a slight generalization 
of restricted volume (\cite{ELMNP}). 

\begin{lem}[Restricted volume in non-compact setup]\label{ncrv}
For a projective morphism $\pi\colon X\to Y$ from 
normal $X$ over a normal 
affine variety $Y$, consider 
a relative ample line bundle $L$ on $X$ and a subscheme 
$Z$ inside a closed fiber $\pi^{-1}(p)$ for a closed 
point $p\in Y$. 
Take normal projective compactifications  
$X\subset \overline{X}$, 
$Y\subset \overline{Y}$, 
and an extension $\overline{\pi}$ of 
$\pi$ as $\overline{X}\to \overline{Y}$. 
Set the divisorial part of $\overline{Y}\setminus Y$ as $D$, which we assume to be an ample Cartier divisor, and 
extension of $L$ 
to $\overline{X}$ as $\overline{L}$ (still relatively ample). 
\begin{enumerate}
\item (well-definedness) \label{resi}
If one considers the restricted volume 
\begin{align}\label{consider.resvol}
\vol_{\overline{X}|Z}(\overline{L}+a\pi^*D) \text{ (cf., \cite{ELMNP})}
\end{align}
for $Z\subset \pi^{-1}(p)$ and $a\in \Q$, the following holds: 
there is a positive rational number $a_0$ such that for any $(\Q_{>0}\ni) a\ge a_0$, the above \eqref{consider.resvol} is constant and 
does not depend on 
the compactification data, i.e., just determined by 
$X,Z,L,Y$ and $\pi$. We denote it simply as 
$$\vol_{X|Z}(L).$$ From the definition, 
one can consider the same for any ($\pi$-ample) 
$\Q$-line bundle $L$. 
\item (uniformity of $a$) \label{resii}
Note that from the proof, $a_0$ 
depends on $L$ and $\overline{L}$. Nevertheless, 
if we fix $\overline{X},\overline{Y},\overline{\pi}$ and 
$L$, 
consider $\overline{L}_1, \overline{L}_2$ on $\overline{X}$ 
and an interval $(c_1,c_2)$ with $\overline{L}_i|_X=L_i$, then $a$ for \eqref{resi} 
can be taken uniformly (i.e., $a_0$ can be taken as a constant) 
for any $L_1\otimes L_2^{\otimes c}$ 
with $c\in (c_1,c_2)$ as far as $L_1\otimes L_2^{\otimes c}$ 
is relatively $\pi$-ample over $Y$. 
\end{enumerate}
\end{lem}

\begin{proof}
Firstly, we prove \eqref{resi}. 
From the definition of the restricted volume (\cite{ELMNP}), if we set $d:=\dim Z$ and 
denote the ideal sheaf for $Z\subset \overline{X}$ as $I_Z$, it follows that 
\begin{align*}
\vol_{\overline{X}|Z}(\overline{L}+a\pi^*D)&:=\limsup_{la\in \Z, l\to \infty}
\dfrac{\dim {\rm Im}(H^0(\overline{X},\overline{L}^{\otimes l}(al\overline{\pi}^*D))\to 
H^0(Z,\overline{L}^{\otimes l}(al\overline{\pi}^*D)|_Z)))}{l^{d}/d!}\\
&=\limsup_{la\in \Z, l\to \infty}
\dfrac{\dim (H^0(\overline{X},\overline{L}^{\otimes l}(al\overline{\pi}^*D))/H^0(\overline{X},I_Z \overline{L}^{\otimes l}(al\overline{\pi}^*D)))}{l^{d}/d!}\\ 
&=\limsup_{la\in \Z, l\to \infty}
\dfrac{\dim (H^0(\overline{Y},((\overline{\pi}_* \overline{L}^{\otimes l})(al D)))/H^0(\overline{Y},(\overline{\pi}_* (I_Z \overline{L}^{\otimes l}))(alD)))}{l^{d}/d!}. 
\end{align*}
Both 
$\overline{\pi}_*\overline{L}^{\otimes l}$ and 
$(\overline{\pi}_* (I_Z \overline{L}^{\otimes l}))$ are coherent for any $l$ and 
we have a short exact sequence 
\begin{align}\label{se:push}
0\to (\overline{\pi}_*(I_Z \overline{L}^{\otimes l}))(alD)
\to (\overline{\pi}_* \overline{L}^{\otimes l})(alD)
\to k(p)^{\oplus r(l)}\to 0 \text{ for }l\gg 0
\end{align}
where $k(p)$ denotes the skyscraper sheaf isomorphic to 
$\k$ and 
$r(l)$ is some positive integer sequence, 
since $R^1\overline{\pi}_*(I_Z \overline{L}^{\otimes l}(alD))=0$ 
for $l\gg 0$. If we take long exact sequence of $H^i(-)$ of 
\eqref{se:push}, we obtain 
\begin{align}\label{se:hi}
0\to H^0(\overline{Y},(\overline{\pi}_*(I_Z \overline{L}^{\otimes l}))(alD))
\to H^0(\overline{Y},(\overline{\pi}_* \overline{L}^{\otimes l})(alD))
\to k(p)^{\oplus r(l)}\to 0 \text{ for }l\gg 0
\end{align}
as far as $H^1(\overline{Y},(\overline{\pi}_*(I_Z \overline{L}^{\otimes l}))(alD))$,  
which holds for $a\gg 0$ since $D$ is ample. 
Hence, the above quantities can be simplified as 
\begin{align*}
&\limsup_{l\to \infty}
\dfrac{\dim (H^0(\overline{Y},(\overline{\pi}_* \overline{L}^{\otimes l})(a l D))/H^0(\overline{Y},(\overline{\pi}_* (I_Z \overline{L}^{\otimes l}))(a lD))}{l^{d}/d!}\\
%=&
%\limsup_{l\to \infty}
%\dfrac{\dim (\overline{\pi}_* \overline{L}^{\otimes l}(a lD)/\pi_* (I_Z %\overline{L}^{\otimes l})(a lD))}{l^{d}/d!}\\
=&
\limsup_{l\to \infty}
\dfrac{\dim (\pi_* L^{\otimes l}/\pi_* (I_Z L^{\otimes l}))}{l^{d}/d!}, 
\end{align*}
where the last equality holds because 
$\overline{\pi}_* \overline{L}^{\otimes l}/\pi_* (I_Z \overline{L}^{\otimes l})$ is supported on $p\in Y$ for any $l$. 

For \eqref{resii}, it follows from the proof of 
above \eqref{resi} as follows. 
Set the normalization of the blow up of $Z\subset \overline{X}$ as 
$\varphi\colon X'={\rm Bl}_{Z}(\overline{X})^{\nu}\to X$ 
with the exceptional Cartier divisor $E:=V((\varphi\circ \nu)^{-1}(I_Z))$. 
We can and do take large enough uniform 
$a$ so that $\varphi^*(\overline{L}_1\otimes \overline{L}_2^{\otimes 
c}(a\overline{\pi}^*D))-\frac{1}{l}(E))$ 
are all relatively ample (as $\Q$-line 
bundle) over $\overline{Y}$ for any $c\in (c_1,c_2)$ and $l\gg 0$. 
This completes the proof. 
\end{proof}
By using above, we rewrite the inspiring notion of 
{\it weighted volume} of Sun-Zhang 
\cite[\S 4]{SZ}, defined 
after its analytic (symplectic) version in 
\cite{CDS}, 
in terms of the restricted volume as 
follows. 
Now, our setup is restricted as follows as 
in \cite[\S 4]{SZ}. 

\begin{Setup}\label{set:Ff}
We take a Fano fibration $\pi\colon X\to Y$ with ${\rm dim}(X)=n$, 
and fix a closed point $p\in Y$. 

Consider general (real valued) valuation $v$ of 
$K(X)$, 
the function field of $X$, whose 
center is inside $\pi^{-1}(p)$ and 
log discrepancy is finite i.e., $A_X(v)<\infty$, 
which we call {\it vertical} valuation (over $p$). 
 
For such $v$, 
\cite[\S 4]{SZ} defines a real-valued invariant which they call 
the {\it weighted volume} 
$\W(v)$. 
We quickly review the original definition of $\W(-)$ in their 
excellent paper and then just give several different expression for further 
works. We first fix $r\in \Z_{>0}$ such that $-rK_X$ is Cartier 
just for convenience of notational convenience, 
and set $L:=\mathcal{O}(-rK_X)$. 
For simpler setup, one can always assume $X$ is smooth and $r=1$, 
and general case has no essential difference of the theory. 
\end{Setup}

\begin{defn}\label{def:DH}
For such $v$, we can define two types of {\it (algebraic) 
Duistermaat-Heckman type measures} 
\footnote{Note that the original 
Duistermaat-Heckman measure \cite{DH} 
was in symplectic geometric setup i.e., 
as a measure on the image of moment maps. 
This is later systematically studied 
in the context of 
K\"ahler geometry, or test configurations 
(\cite{Don02}), 
notably by Hisamoto 
\cite{Hisamoto1, Hisamoto2} and further in 
algebro-geometric setup in 
Boucksom-Hisamoto-Jonsson \cite{BHJ}.} 
on $\R$: 
%both encoding the distribution of 
%$\dim R_{l,\vec{m}}$ for $l\in \Z$, $\vec{m}\in M$: 
\begin{enumerate}
\item \label{def:DH1} 
(Fiber type: cf., \cite{SZ}) ${\rm DH}(v)={\rm DH}_f(v)$ 
is the limit of the following quantized version. 
For $l\in \Z_{>0}$, we  define 
\begin{align}
&{\rm DH}_l(v)={\rm DH}_{f,l}(v)\\
:=&\frac{1}{(lr)^n}\sum_{x\ge 0}\dim (\mathcal{F}_v^{x lr}H^0(-lrK_X)/\mathcal{F}_v^{>x lr}H^0(-lrK_X)) \delta_x.
\end{align}
Here, $\delta_x$ denotes the Dirac measure supported at $x\in \R$, and 
$\F_v$ is defined as a filtration 
(cf., Definition \ref{def:1.3} for the general definition) 
\begin{align}\label{67}
&\mathcal{F}_v^{x lr}H^0(-lrK_X):=\{s\in H^0(-lrK_X)\mid v(s)\ge x lr\}\\
&\mathcal{F}_v^{>x lr}H^0(-lrK_X):=\{s\in H^0(-lrK_X)\mid v(s)> x  lr\}. 
\end{align}
This is a certain variant of 
(rescaled) weight measure in the sense of \cite[1.5]{BHJ}. 
Then, we consider the limit measure 
$${\rm DH}(v):={\rm DH}_f(v):=\lim_{l\to \infty}{\rm DH}_{f,l}(v).$$
\cite{SZ} uses this for their definition of weighted volume. 
Following it, we use this for a while. 

We sometimes denote the above measure ${\rm DH}(v)$ as 
${\rm DH}_X(v)$ or ${\rm DH}_{X,f}(v)$, to avoid confusion. 

\item \label{def:DH2} (Base type: cf., \cite[\S 2, 2.17]{Od24a}) 
As in \cite[Definition 2.5]{SZ}, 
suppose further 
$X$ and $Y$ are both given algebraic 
actions of an algebraic (split) torus $T\simeq \G_m^r$, 
so that 
\begin{itemize}
\item $T$-action on $Y$ is good and $\xi \in N\otimes \R$ is a positive vector field ((abstract) Reeb vector field) cf., e.g., \cite{CS}, \cite[\S 2]{Od24a}). 
\item the Fano fibration morphism $\pi\colon X\to Y$ is $T$-equivariant. 
\end{itemize}
We denote $\Gamma(\pi_* \mathcal{O}_X(-lrK_X))$ 
as  $R_l$ for $l\in \Z_{\ge 0}$ 
and its $\vec{m}$-eigen subspace (with respect to the 
$T$-action) as $R_{l,\vec{m}}$ for $\vec{m}\in M:={\rm Hom}(T,\G_m)$ i.e., with any character $\vec{m}$ of $T$. 
Then, we can define 
another (base type) Duistermaat-Heckman measure 
${\rm DH}_b$ as in \cite[Definition 2.17]{Od24a}. For its definition, 
we first fix range of $\vec{m}$ and then consider Dirac type measures, 
and then take limit probability measure ${\rm DH}_b$ on $\R$ 
defined as 
$$\lim_{c\to \infty}
\biggl(\sum_{\substack{l\in \Z_{\ge 0}, \\ 
\vec{m}\in M\setminus \{\vec{0}\}, \langle 
\vec{m},\xi\rangle <c}}\dfrac{\dim(R_{l,\vec{m}})}
{\sum_{\vec{m}\in M\setminus \{\vec{0}\}, \langle 
\vec{m},\xi\rangle <c}\dim(R_{\vec{m}})}\quad \delta_{\frac
{l}{\langle m,\xi 
\rangle}}\biggr).$$
We omit the details for now. See Definition \ref{def:eqff} and later discussions for further studies on this setup, which do not really use 
the above base type Duistermaat-Heckman measure yet. 
\end{enumerate}
\end{defn}

We generalize the above Duistermaat-Heckman measure of fiber type i.e., 
Definition \ref{def:DH} \eqref{def:DH1} as follows. 

\begin{defn}\label{def:1.3}
In the above Setup \ref{set:Ff}, in this paper, 
\begin{enumerate}
    \item {\it (vertical) filtration} $\mathcal{F}^\bullet$ of $\{\pi_*\O_X(-lrK_X)\}_{l\in \Z_{>0}}$ means 
    the sub-indexed data of %coherent 
    $\mathcal{O}_Y$-modules 
    $\mathcal{F}^{x lr}\pi_*\O_X(-lrK_X)$ for each $x\in \R_{\ge 0}$ such that 
    \begin{itemize}
    \item $(\mathcal{F}^{x lr}\pi_*\O_X(-lrK_X))|_{Y\setminus p}
    =\pi_*\O_X(-lrK_X))|_{Y\setminus p}$, 
        \item $\mathcal{F}^{x' lr}\pi_*\O_X(-lrK_X)\subset  \mathcal{F}^{x lr}\pi_*\O_X(-lrK_X)$ if $x'\ge x$, 
        \item $\mathcal{F}^{x lr}\pi_*\O_X(-lrK_X)\cdot 
        \mathcal{F}^{x l'r}\pi_*\O_X(-l'rK_X)\subset 
        \mathcal{F}^{x lr}\pi_*\O_X(-(l+l')rK_X)$. 
        We suppose this $\F$ contains 
        $F_v$ of the form \eqref{67} for some (vertical) $v$ 
        and set its {\it (relative) volume function}
        \footnote{note that this corresponds to minus the volume function of that of \cite{HanLi}, to match with the convention of \cite{Li} and earlier works}
        as 
        \begin{align}\label{vollambda}
        \vol\mathcal{F}^x:=\lim_{l\to \infty}\dim(\pi_*\O_X(-lrK_X)/\F^{x rl
        }\pi_*\O_X(-lrK_X)) \quad (\in \R_{\ge 0}). 
        \end{align}
        Here, $c$ is some fixed real constant. 
        By \cite[Appendix]{SZ} 
        using the Okounkov body (cf., also 
        \cite{BC}), combining with our Lemma \ref{ncrv}, 
        $$\frac{1}{(rl)^n}\frac{d}{dx}\dim(\pi_*L^{\otimes l}/\F^{x lr}\pi_*L^{\otimes l})$$ as a distribution weakly 
        converges to some measure on $\R$ for $l\to \infty$, 
        which we denote as ${\rm DH}_X(\F)={\rm DH}(\F)$ 
        and call it {\it Duistermaat-Heckman measure (of fiber type)} 
        for $\mathcal{F}$. 
    \end{itemize}
    \item {\it (vertical) ideal} $I\subset \O_X$ means an ideal sheaf 
    such that $I|_{\pi^{-1}(Y\setminus p)}=\O_{X}|_{\pi^{-1}(Y\setminus p)}$. 
    \item {\it (vertical) graded ideals} $\{I_l\}_{l\in \Z_{\ge 0}}$ means 
    that $I_l$ is vertical ideal of $\O_X$ such that $I_l\cdot I_{l'}\subset I_{l+l'}$ for any $l, l'\in \Z_{\ge 0}$. 
    Note that there is naturally associated 
    (vertical) filtration defined as 
    $\mathcal{F}^{x lr}\pi_*\mathcal{O}_X(-lrK_X):=
    \pi_*(I_l^{\lceil xlr \rceil}\cdot \mathcal{O}_X(
    \lfloor -lrK_X\rfloor )).$ We denote that as $\mathcal{F}_{I_\bullet}$. 
    \item 
    {\it (vertical) ideal $I$ with exponent $m$} simply  means that $I$ is a vertical ideal of $\O_X$ and 
    $m\in \Z_{>0}$. For that, we define 
    a vertical filtration $\F_{I,m}$ as 
    $\F_{I,m}\pi_* L^{\otimes l}:=\pi_* (I^{\lceil \frac{l}{m}\rceil}L^{\otimes l})$. 
\end{enumerate}

\end{defn}

\begin{defn}[{\cite[\S 4]{SZ}}]
In the above Setup \ref{set:Ff}, 
the {\it weighted volume} for above $v$ by Sun-Zhang is defined as 
$$\W(v):=e^{A_X(v)}\cdot \int_{\R_{\ge 0}\ni x}e^{-x}{\rm DH}(v),$$
where ${\rm DH}(v)$ is as Definition \ref{def:DH} \eqref{def:DH1}. 
Similarly, for (vertical) graded ideals $I_\bullet=\{I_l\}_l$, 
we can define its weighted volume 
$$\W(I_{\bullet}):=e^{\lcth(X;I_\bullet)}\cdot \int_{\R_{\ge 0}} e^{-x}{\rm DH}(\mathcal{F}_{I_\bullet}).$$
As a special case, 
$$\W(I):=e^{\lcth(X;I)}\cdot \int_{\R_{\ge 0}} e^{-x}{\rm DH}(\mathcal{F}_{I}).$$
These ${\rm DH}$ are as defined in Definition \ref{def:1.3}. 
By integration by part, using the 
good properties of 
``weight" function $e^{-x}$, 
we also have different expression as 
\begin{align}
\log\W(v)&=A_X(v)+\log\int_0^\infty e^{-x}\vol\F_v^x dx,\\
\log\W(I_\bullet)&=\lcth(X;I_\bullet)+\log
\int_0^\infty e^{-x}\vol\F_{I_\bullet}^x dx. 
\end{align}
\end{defn}

Then, \cite{SZ} defines the weighted volume of Fano fibration as follows: 

\begin{defn}[Weighted volume of Fano fibration {\cite[Definition 6.5]{SZ}}]
For a Fano fibration $X\xrightarrow{\pi}Y$, 
the weighted volume means 
\begin{align}\label{def:wv}
\W(\pi):=\inf_v \W(v),
\end{align}
where $v$ runs over all (real valued) valuation 
whose center is supported inside $\pi^{-1}(p)$ 
and $A_X(v)<\infty$ (Setup \ref{set:Ff}). 
Note that the 
Conjecture \ref{conj:SZ2step} ($=$ review of 
\cite[Conjecture 6.4]{SZ}) 
would imply that the infimum is actually the minimum. 
\end{defn}
The following viewpoint (cf., \cite[\S 2]{Od24a}) is important for our purpose. 
\begin{lem}[As degeneration]\label{asdegen}
There is a rational polyhedral cone $\sigma$ of $N\otimes \R$ and corresponding 
affine toric variety $U_\sigma$ with its (unique) $T$-invariant closed point 
$p_\sigma$, 
together with $T$-equivariant morphisms 
$\Pi_\sigma\colon \mathcal{X}_\sigma\to \Y_\sigma\xrightarrow{f_\sigma} 
U_\sigma$ 
whose general fibers are $X\xrightarrow{\pi}Y$ and the closed 
fiber over $p_\sigma$ i.e., 
$T\curvearrowright (\Pi_\sigma^{-1}(p_\sigma)\to f_\sigma^{-1}(p_\sigma))$ 
is $T\curvearrowright (X_v\to Y_v)$. 
\end{lem}

\begin{proof}
It follows from the same arguments as \cite[Theorem 2.11]{Od24a} 
(also cf., the references therein: \cite{Teissier, LX}). 
\end{proof}

\begin{prop}\label{prop:1.7}
   $\W(\pi)$ can be re-written as follows: 
\begin{align}
 \label{8}   \W(\pi)&=\inf_{\{I_\bullet\} :\text{ vertical} }e^{\lcth(X;I_\bullet)}
    \int_{\R_{\ge 0}} e^{-x}{\rm DH}(\mathcal{F}_{I_\bullet})\\ 
 \label{9}   &=\inf_{I :\text{ vertical}, \hspace{1mm}m\in \Z_{>0} }e^{m\lcth(X;I)}
    \int_{\R_{\ge 0}} e^{-x}{\rm DH}(\mathcal{F}_{I,m}). 
\end{align}
Moreover, $\W(v)$ can be also written as 
\begin{align}\label{def:wv2}
\W(\pi)=\inf_{v\colon {\text{ divisorial}}}\W(v), 
\end{align}
where $v$ runs over only divisorial valuations 
in the following sense: 
of the form $v=\frac{{\rm ord}_E}{b}$ 
where $E\subset Y\xrightarrow{\varphi} X$ 
is some blow up with exceptional prime divisor $E$ 
and $b$ is a positive rational number. 
\end{prop}

\begin{proof}
$\ge$ of \eqref{8} and $\le$ of \eqref{9} follows from the definitions and 
$\lcth(X;I_\bullet)=\lim_{m\to \infty}
m \lcth(X;I_m)$ (cf., \cite{JM}). 
$\ge$ of \eqref{9} follows from the standard approximation 
by using each $I_l$ and \cite{JM} again. 

The remaining task for the proof of \eqref{8}, 
\eqref{9} 
is to confirm that for any given $I$, there is a 
valuation $v$ of $K(X)$ 
centered inside $\pi^{-1}(p)$ that proves the $\le$ 
side of \eqref{8}. For that, 
one can assume that $m=1$ as otherwise 
multiply $m$ to the obtained valuation 
in general case. Since 
$\lcth(X;I)=\min_E \frac{A_X(E)}{\mult_E(I)}$, 
one can take the minimizer $E$ of the right 
hand side and set $v:=\frac{{\rm ord}(E)}
{\mult_E(I)}$. Then, 
$\F_v\supset \F_{I}$ so that 
$\vol(\F_v^x)\le \vol(\F_I^{x})$. 
Hence, we obtain the desired inequality. 

The $\le$ direction of \eqref{def:wv2} 
is obvious while $\ge$ direction 
follows from the last part of the above arguments 
which says that for any vertical ideal $I$ with 
exponent $m$, there is a (vertical) prime divisor 
$E$ such that for $v_E:=\frac{{\rm ord}(E)}
{\mult_E(I)}$, we have 
$$e^{A_X(v_E)}\cdot \int e^{-x} {\rm DH}(\F_v)
\le e^{m\lcth(I)}\cdot \int e^{-x}\DH(\F_{I,m}).$$
\end{proof}

Following above proposition, now we focus on the 
divisorial valuation case. 

\begin{prop}[with divisors {\it over} $X$]\label{DHcal}
In the above Setup \ref{set:Ff}, 
firstly we consider 
(rescaled) divisorial valuation $$v:=\frac{\ord_E}{b}$$ of $K(X)$, 
where $b\in \R_{>0}$ and $E$ is a 
divisor $E$ over $X$ as realized in a 
($\pi^{-1}(Y\setminus p)$-admissible) blow up 
$\varphi\colon X'\to X$ of $X$, with normal $X'$. 
We denote $q\in \pi^{-1}(p)$ as the center of $\ord_E$ 
i.e., the generic point of $\varphi(E)$. 
Then, the  following holds:

\begin{enumerate}
\item 
In this divisorial case, 
the (fiber type) Duistermaat-Heckman measure can be written as 
$${\rm DH}(v)=b\vol_{X'|E}(-\varphi^*K_X-bxE)dx.$$ 

\item \label{DHcal(ii)}
The weighted volume $\W(v)$ of 
\cite[Definition 4.2]{SZ} is 
\begin{align}\label{wv1}
&e^{A_X(v)}\cdot \int_{\R_{\ge 0}} b\cdot e^{-x}{\rm vol}_{X'|E}(-\varphi^*K_X-bxE)dx\\ 
&=b (e^{\frac{A_X(E)}{b}})\cdot \int_{\R_{\ge 0}}  e^{-x}{\rm vol}_{X'|E}(-\varphi^*K_X-bxE)dx. 
\end{align}
If $Y$ is a point i.e., $X$ is Fano variety, 
it can be also written as 
$$b (e^{\frac{A_X(E)}{b}})\cdot 
\biggl((-K_X)^n-
\int_{\R_{\ge 0}}  e^{-x}{\rm vol}(-\varphi^*K_X-bxE)dx\biggr).$$ 
It recovers 
the $\tilde{\beta}$-invariant of Han-Li \cite{HanLi} (see 
\cite[Example 4.5]{SZ}). 
\item \label{Jensen1}
We have 
\begin{align}
\label{wv2}\log \W(v)&= A_X(v)+\log \int_{\R_{\ge 0}}e^{-x} b\cdot 
\vol_{X'|E}(-\varphi^*K_X-bxE)dx\\
\label{wv2.5}&= \frac{A_X(E)}{b}+\log \int_{\R_{\ge 0}}e^{-x} b \vol_{X'|E}(-\varphi^*K_X-bxE)dx\\
&\label{wv3}\ge \frac{A_X(E)}{b}+\log b -(1-c)
\dfrac{\int_{\R_{\ge 0}}x e^{-cx} \vol_{X'|E}(-\varphi^*K_X-bxE)dx}{\int_{\R_{\ge 0}}e^{-cx}\vol_{X'|E}(-\varphi^*K_X-bxE)dx}, 
\end{align} 
for any $c\in (0,1)$. (If $Y$ is a point, one can take $c=0$ as well so that the right hand side is simpler.) 
\item \label{Jensen2}
In the same setup, similarly, for any $s_1, (\le) s_2\in \R_{\ge 0}$, 
%if we set $\int_{s_1}^{s_2}{\rm DH}(v)dx=:V_{s_1,s_2}$, 
then 
we also have another lower bound: 
\begin{align}
\label{wv7}\log \W(v)&= A_X(v)+\log \int_{\R_{\ge 0}}e^{-x} b\cdot 
\vol_{X'|E}(-\varphi^*K_X-bxE)dx\\
&\label{wv8}\ge \frac{A_X(E)}{b}-
\dfrac{\int_{s_1}^{s_2}x \vol_{X'|E}(-\varphi^*K_X-bxE)dx}{\left(\int_{s_1}^{s_2}\vol_{X'|E}(-\varphi^*K_X-bxE)dx\right)}\\
&+\log \left(\int_{s_1}^{s_2}\vol_{X'|E}(-\varphi^*K_X-bxE)dx\right)+\log b. 
\end{align}

\item (Laplace transform equation) \label{min.b}
If $v=\frac{{\rm ord}_E}{b}$ minimizes the weighted volume $\W(v)$, 
$b$ satisfies a vanishing of certain Laplace transform: 
\begin{align}
\int_{\R_{\ge 0}}e^{-\frac{1}{b}y}\cdot (y-A_X(v)){\rm vol}_{X'|E}(-K_{X'}-yE)dy=0. 
\end{align}
More generally, if (not necessarily divisorial) valuation $v$ minimizes 
$v$, then it satisfies a similar equation: suppose that the 
density function of ${\rm DH}(v)$ is $\mathcal{R}_v(y)$ i.e., 
${\rm DH}(v)=\mathcal{R}_v(y)dy$. Then, if we put 
$\tilde{\mathcal{R}}_v(y):=((y-A_X(v))\mathcal{R}_v(y))$, we have 
\begin{align}
\int_{\R_{\ge 0}}e^{-y}\cdot \tilde{\mathcal{R}}_v(y) dy=0. 
\end{align}

\end{enumerate}
\end{prop}
\begin{proof}
For simplicity of notation, we suppose $K_X$ is Cartier below. 
Otherwise, we consider $\Q$-Gorenstein index 
$r$ with $rK_X$ Cartier and run the same arguments. 

We prove the first item as follows. 
Recall again from the original \cite[\S 4]{SZ}  
defines the quantized version of their (fiber type) 
Duistermaat-Heckman measure is 
${\rm DH}_l(v)=\frac{1}{l^n}\sum_{x\ge 0}\dim (\mathcal{F}_v^{x l}H^0(-lK_X)/\mathcal{F}_v^{>x l}H^0(-lK_X)) \delta_x$. 
Here, $\mathcal{F}_v$ is a decreasing filtration such that 
\begin{align*}
&\mathcal{F}_v^{x l}H^0(-lK_X)
:=\{f\in H^0(-lK_X)\mid v(f)\ge x\} \text{ and}\\ 
&\mathcal{F}_v^{>x l}H^0(-lK_X):=\{f\in H^0(-lK_X)\mid v(f)> 
x\},
\end{align*}
and $\delta_x$ denotes the Dirac measure supported on 
$x\in \R$. 
Note that $v(f)$ is defined via using the local trivialization of 
$K_X$ around the center of $v$. 
In other words, we have 
\begin{align*}
\mathcal{F}_v^{x l}H^0(-lK_X)
&=(\pi\circ \varphi)_* \mathcal{O}_X(-l\varphi^*K_X-\lfloor bx 
l\rfloor E) \text{ and}\\ 
\mathcal{F}_v^{>x l}H^0(-lK_X)&=(\pi\circ 
\varphi)_*\mathcal{O}_X(-l\varphi^*K_X-(\lfloor bx l \rfloor +1)
E). 
\end{align*}
Then, Sun-Zhang 
\cite[Proposition 4.1, Appendix A]{SZ} proves that this 
weakly converges to a limit measure ${\rm DH}(v)$. 
Thus, for $x, \epsilon \in \Q_{>0}$, 
we have 
\begin{align*}
&{\rm DH}(v)(x,x+\epsilon)\\ 
&=\limsup_{l\to \infty}
\dfrac{\dim(H^0(X,-l\varphi^*K_X-\lceil bx l\rceil E)/H^0(X,-m\varphi^*K_X-\lfloor b(x+\epsilon)m\rfloor E))}{l^n/n!}\\
&=\limsup_{l\to \infty}
\dfrac{\dim(H^0(\overline{X},-l\varphi^*K_{\overline{X}}-\lceil bx l\rceil E +la\overline{\pi}^*D)/H^0(-l\varphi^*K_{\overline{X}}-\lfloor b(x+\epsilon)l\rfloor E+la\overline{\pi}^*D))}{l^n/n!} \\
%\text{ for }l\gg 0\\ 
&=\limsup_{l\to \infty}
\dfrac{\dim(H^0(\overline{X},-l\varphi^*K_{\overline{X}}-\lceil bx l\rceil E +la\overline{\pi}^*D)/H^0(-l\varphi^*K_{\overline{X}}-\lceil bx l\rceil E+la\overline{\pi}^*D))}{l^n/n!} \\
%\text{ for }l\gg 0\\ 
&=\lim_{l\to \infty}
\dfrac{\dim(H^0(\overline{X},-l\varphi^*K_{\overline{X}}- \lceil bx l\rceil E +la\overline{\pi}^*D)/H^0(-l\varphi^*K_{\overline{X}}-\lfloor b(x+\epsilon)l\rfloor  E+la\overline{\pi}^*D))}{l^n/n!} \\ 
%\text{ for }m\gg 0 
&(\text{cf., \cite{LM}})\\ 
&=\vol(-\varphi^*K_{\overline{X}}-bx E+a\overline{\pi}^*D)-
\vol(-\varphi^*K_{\overline{X}}- b(x+\epsilon) E+a\overline{\pi}^*D). 
\end{align*}
Combined with above, if we use \cite[Corollary C]{LM} and 
\cite[Corollary C]{BFJd}, it follows that 
\begin{align}\label{DHformula}
{\rm DH}(v)(x,x+\epsilon)= b\int_x^{x+\epsilon}\vol_{X'|E}(-\varphi^*K_X-bxE)dx.
\end{align}
Therefore, the original definition 
\cite[Definition 4.2]{SZ} gives the first assertion. 
The second item of the above proposition is simply a corollary to the 
first item, simply combined with the integration by part at the end. 

\eqref{Jensen1} then follows from the Jensen's inequality 
with respect to the convex function $e^{-(1-c)x}$ and the probability measure 
$\dfrac{e^{-cx} \vol_{X'|E}(-\varphi^*K_E-bxE)}{\int_{x=0}^\infty e^{-cx}\vol_{X'|E}(-\varphi^*K_X-bxX)dx}$. Note that the denominator is 
finite as $c>0$ (or $Y$ is a point). \eqref{Jensen2} 
also similarly follows from the Jensen's inequality. The last assertion is 
straightforward from standard calculation. 
\end{proof}

In particular case when the divisor $E$ is {\it on} $X$, 
as a simple  consequence of Birkar-Cascini-Hacon-Mckernan \cite{BCHM} and 
standard 
calculations, as in \cite{Fjt.div}, we have the following more explicit 
description. 

\begin{prop}[with divisors {\it on} $X$]\label{DHcal2}
Suppose the base field $k$ is of characteristic $0$. 
Consider the case when $\varphi$ can be taken as identity i.e., when the 
prime divisor $E$ is a divisor of $X$. Then, the following hold: 
\begin{enumerate}
\item \label{divcase1}
there is a increasing finite sequence of positive rational numbers 
$0=\tau_0<\tau_1<\cdots<\tau_m=\tau(E)$ and a finite birational contractions 
$\phi_i\colon X\dashrightarrow X_i\to Y$ which are 
the ample models of $-\varphi^*K_X-xE$ for any $x\in (\tau_{i-1},\tau_i)\cap \Q$, 
in the sense of e.g., \cite[3.6.5]{BCHM}. 

\item \label{divcase2}
If we set the strict transform of $E$ on $X_i$ as $E_i$, 
then $$\vol_{X'|E}(-\varphi^*K_X-bxE)=((-K_{X_i}-bxE_i)^{n-1}.E_i).$$
Hence, the 
density function $\mathcal{R}_v(y)$ of ${\rm DH}(v)$ is (cf., 
Proposition \ref{DHcal} \eqref{min.b}): 
\begin{align*}
\mathcal{R}_E(x)&:=\mathcal{R}_{\rm ord(E)}(x)\\ 
&=((-K_{X_i}-xE_i)^{n-1}.E_i) \text{ if }\tau_{i-1}\le x\le \tau_i, \\  
\mathcal{R}_{bE}(x)&:=\mathcal{R}_{b{\rm ord(E)}}(x)\\ 
&=b((-K_{X_i}-bxE_i)^{n-1}.E_i) \text{ if }\tau_{i-1}\le bx\le \tau_i, 
\end{align*}
for each $i$, 
so that 
\begin{align}
\W\biggl(\frac{{\rm ord}_E}{b}\biggr)&=b(e^{\frac{1}{b}})\cdot \biggl(\sum_{i=1}^m \int_{\tau_{i-1}/b}^{\tau_i/b} e^{-x}((-K_{X_i}-bxE_i)^{n-1}.E_i)dx\biggr)
\\ 
&\label{simpleineq}\ge e \cdot \biggl(\sum_{i=1}^m \int_{\tau_{i-1}/b}^{\tau_i/b} e^{-x}((-K_{X_i}-bxE_i)^{n-1}.E_i)dx\biggr).%\ge e. 
\end{align}
\item \label{divcase3}
If $\W\bigl(\frac{{\rm ord}_E}{b}\bigr)$ is minimized at 
$b$ (while fixing $E$), then $c:=\frac{1}{b}$ satisfies the vanishing of 
Laplace transform of some rational piece-wise polynomial: 
\begin{align}
&\int_{\R_{\ge 0}} e^{-cy}\cdot (y-A_X(E))((-K_{X_i}-yE_i)^{n-1}.E_i) dy\\=
&\sum_{i=0}^m \int_{\tau_{i-1}}^{\tau_i}(a_n(i)y^n+\cdots+a_0(i)) e^{-cy}\\
=&0,
\end{align}
where $a_n(i),\cdots,a_0(i)\in \Q$ so that 
$$(a_n(i)y^n+\cdots+a_0(i))=(y-A_X(E))((-K_{X_i}-yE_i)^{n-1}.E_i)$$ 
for each $i$. 
\item (Via Gamma function and rational  exponential polynomial) \label{divcase4}
In particular, in the case \eqref{divcase3}, 
$c$ satisfies some equation in terms of (incomplete) Gamma functions 
$\Gamma(m,-)$ and $\gamma(m,-)$ ($m\in \Z$). 
Thus, 
the possible value of $\min_{b\in \R_{>0}} \W(b{\rm ord}_E)$ 
has only countable possibilities for the setup $E\subset X$. 

In particular, there is an integral 
exponential polynomial $F(X)\in \Z[X,e^X]$ 
and $a\in \Z_{>0}$ such that 
$F(c^a)=0$ for the minimizing point $c$ (which 
exists), and 
$\W(c\cdot \ord_E)$ for that $c$ is a 
finite sum of numbers of the form 
$f_i(c)\cdot e^{r_i\cdot c} (r_i\in \Q, 
f_i\in \Q[t,t^{-1}]$. 

\end{enumerate}
\end{prop}
\begin{proof}
\eqref{divcase1} and the former half (until the equality) 
\eqref{divcase2} are easy. 
Indeed, the existence of finite ample models $\phi_i\colon X\dashrightarrow X_i$ follow from 
\cite[1.3.2]{BCHM} 
and the rest of the proof is straightforward (see \cite[\S 2, \S 5, \S 8]{Fjt.div}). The latter half of \eqref{divcase2} i.e., 
\eqref{simpleineq} follows from 
standard minimizer calculation of the term $be^{1/b}$, as achieved at 
$b=1$ and the monotone increaseness of the intersection numbers (or the restricted volume function). 
The remained \eqref{divcase3} follows 
similarly as Proposition \ref{DHcal} \eqref{min.b} and 
reduction to 
basic integral of $\int e^{-x}x^k (k\in \Z_{\ge 0})$ gives the former half of 
\eqref{divcase4}. The latter half of 
\eqref{divcase4} follows from the integration by parts and the previous expression of 
${\rm DH}(c\cdot \ord_E)$. 
\end{proof}

\begin{Rem}[For other weights functions case]
There are some works of generalizations of K\"ahler-Einstein metrics 
of self-similar solution (soliton) type by 
Mabuchi \cite{Msol1, Msol2}, 
Berman-Nystrom \cite{BWN}, 
Han-Li \cite{HanLi} and Apostolov-Lahdili-Legendre 
\cite{ALL}. These correspond to other or 
general weight function $v$ on the moment polytope. 
Recall that compact K\"ahler-Ricci soliton case 
corresponds to the case  
$v=e^{-x}$ for some linear function $x$, 
which is an origin of the weight function $e^{-x}$. 

On the other hand, 
as it is obvious from our above discussions, 
many parts of our arguments for Sun-Zhang theory \cite{SZ} for {\it non-compact} K\"ahler-Ricci solitons 
in this paper focus on 
the Duistermaat-Heckman type measure and do {\it not} 
use the properties of the exponential function so often. 
Hence, we naturally expect that our analysis give some 
extension in more generalized setup in future.     
\end{Rem}

%%%%%%%%

\subsection{Equivariant Fano fibrations}

We now focus on torus-equivariant Fano fibrations, as we briefly 
introduced in Definition \ref{def:DH} \eqref{def:DH2}. 
First we recall the setup again after \cite[\S 2, \S 5]{SZ}. 

Here, $N$ is a lattice (free finitely generated abelian group), 
$M$ is its dual lattice, $\xi\in N\otimes_{\Z}\R$, and $T:=N\otimes_\Z \G_m$ 
is the split algebraic $\k$-torus. 

\begin{defn}[{\cite[Definition 2.5]{SZ}}]\label{def:eqff}
In this paper, a {\it $(T\ni \xi)$-equivariant 
Fano fibration} or simply {\it $\xi$-equivariant Fano fibration} 
(originally called polarized\footnote{The usage of the term ``polarization" here originates from its usage in the context of 
Sasaki-Einstein geometry (cf., e.g., \cite{CS}). Note that if $\xi$ is 
rational, then the quotient of $Y\setminus y$ has a natural 
(pluri-anticanonical) polarization.} 
Fano fibration in \cite[2.5]{SZ}) 
refers to a Fano fibration 
$\pi\colon X\to Y$ with 
equivariant torus $T$-actions on $X,Y$ 
such that 
$T\curvearrowright Y$ is a good action, 
together with a choice 
$\xi\in N\otimes \R$ which gives a 
(abstract) Reeb vector field 
(positive 
vector field) of $Y$ (see e.g., \cite{CS, Od24b}). 
\end{defn}

Note that in this equivariant setup, the weighted volume has 
the following expression. 
Take $\xi \in N\otimes \R$ and the associated valuation 
$v_\xi$ (\cite[\S 5.1]{SZ}). 
Then, 
the weighted volume $\W(\xi)$ in this situation 
can be written as 
(cf., \cite[(4.5), also cf., \S 5 (5.7, Appendix B)]{SZ}): 
\begin{align}\label{Winfsum}
\W(v_\xi)=-\lim_{l\to \infty}\frac{1}{(rl)^n}
\sum_{\vec{m}\in M}  e^{-\langle\frac{\vec{m}}{rl},\xi \rangle}\dim R_{l,\vec{m}}.
\end{align}
Each term of the right hand side 
is a ``quantized" analogue of the weighted volume, 
which absolutely converges 
due to the sub-polynomial divergence order of $\dim R_{l,\vec{m}}$, 
as a standard fact 
(cf., e.g., \cite[5.8.19]{KR}, 
\cite[Lemma 4.2 (and its proof)]{CS}, \cite[A.17, B.2]{SZ}). 
As the original \cite[4.1, (4.5), 5.7, Appendix A, B]{SZ} 
(essentially) explains, the above equality \eqref{Winfsum} follows 
almost from its definition. 

Then \cite{SZ} defines K-stability of the above 
concepts, generalizing and unifying \cite{Don02, CS, BWN, HanLi, BLXZ}. 

\begin{defn}[K-stability of equivariant Fano fibrations {\cite[\S 5.2]{SZ}}]
\label{review:Kst}
\begin{enumerate}
\item 
A {\it test configuration} of  $\xi$-equivariant 
Fano fibration $X\to Y$ is a set of following data: 
\begin{enumerate}
\item 
a quasi-projective variety $\mathcal{X}$ with 
its ample line bundle $\mathcal{L}$ 
and an affine variety $\mathcal{Y}$
\item 
morphisms 
$\mathcal{X}\xrightarrow{\Pi} \mathcal{Y}\xrightarrow{\Pi_\Y}\mathbb{A}^1$ 
with $\Pi_\X:=\Pi_\Y\circ \Pi$ 
such that $\Pi_Y$ is flat and surjective 
\item 
$T$-action on $\mathcal{X}$ equipped with its 
linearization on $\mathcal{L}$, 
\item 
$T$-action on $\Y$ (and trivial action on $\A^1$), 
which is $T$-equivariantly faithfully flat in the sense of 
\cite[\S 2]{Od24a}, 

\item 
$\G_m$-action on $(\X,\mathcal{L})$, $\Y, \A^1$ 
(the last with weight $1$). We denote this action 
sloppily as $\eta$, following \cite{SZ}. 

\end{enumerate}
such that $\Pi, \Pi_\X,\Pi_\Y$ are all 
$T\times \G_m$-equivariant. Further, if 
denote the fibers $X_t:=\Pi_\X^{-1}(t)$ and $Y_t:=\Pi_\Y^{-1}(t)$, 
general fiber $X_t\to Y_t$ i.e., $t\neq 0$ case 
are all isomorphic to $X\to Y$. 
There is a natural trivial compactification of 
$\X\to \Y$ over $\P^1$, 
by adding a trivial fiber ($\simeq (X\to Y))$) 
which we denote as 
$(\overline{\X},\overline{\mathcal{L}})\xrightarrow{\overline{\Pi}
} \overline{\Y}\xrightarrow{\overline{\Pi}_\Y} \P^1$. 
We set $\overline{\Pi}_\X:=\overline{\Pi}\circ 
\overline{\Pi}_\Y$. 

\item (\cite[5.2]{SZ}) 
A {\it special test configuration} of $\xi$-equivariant 
Fano fibration $X\to Y$
refers to the special case of test configurations 
when 
$(\mathcal{X}, \X_0)$ is purely log 
terminal and 
$\mathcal{L}\simeq \mathcal{O}_{\X}(-r'K_{\X})$ with some 
$r'\in \Z_{>0}$. 
Note that then 
each ``fiber" $(T\curvearrowright (X_t\twoheadrightarrow Y_t),\xi)$ is a 
$\xi$-equivariant Fano fibration, even when $t=0$. 

For any special test configuration $\X\to \Y$, we define the 
{\it Donaldson-Futaki invariant} as 
  $${\rm DF}(\Pi)=\dfrac{d}{dt}|_{t=0}\W_{t=0}(\xi+t\eta).$$
\item 

Let us decompose 
$(\Pi_\X)_*\mathcal{L}^{\otimes l}$ by the $T$-action on it 
to its eigensubsheaves 
as $\oplus_{\vec{m}\in M}
(\Pi_\X)_*\mathcal{L}^{\otimes l})_{\vec{m}}$. 
We also set 
$$((\overline{\Pi}_\X)_*\overline{\mathcal{L}}^{\otimes l})_{\xi,s}
:=\oplus_{\vec{m}\in M, \langle \vec{m},\xi\rangle=ls} ((\overline{\Pi}_\X)_*\overline{\mathcal{L}}^{\otimes l})_{\vec{m}},$$ 
for $s\in \R$. 
These are all locally free 
coherent sheaves over $\A^1$. 
We define its extensions 
$\oplus_{\vec{m}\in M}
((\overline{\Pi}_\X)_*\overline{\mathcal{L}}^{\otimes l})_{\vec{m}}$ 
and 
$((\overline{\Pi}_\X)_*\overline{\mathcal{L}}^{\otimes l})_{\xi,s}$ 
similarly by using the above-mentioned compactification 
$(\overline{\X},\overline{\mathcal{L}})\xrightarrow{\overline{\Pi}
} \overline{\Y}\xrightarrow{\overline{\Pi}_\Y} \P^1$.  
For each $l\in \Z_{\ge 0}$, 
we consider 
$\sum_{t\in \R}e^{-
t}\deg((\overline{\Pi}_\X)_*\overline{\mathcal{L}}^{\otimes l})_{\xi,t}$. 
Note that 
$\{s\in \R\mid ((\Pi_\X)_*\mathcal{L}^{\otimes 
l})_{\xi,s}\neq 0\}$ 
is discrete and we believe 
$\sum_{s\in \R}e^{-
t}\deg((\overline{\Pi}_\X)_*\overline{\mathcal{L}}^{\otimes l})_{\xi,ts}$ 
for each $t, l$, 
we can define generalized Donaldson-Futaki invariant appropriately. 
We leave the details as future problem. 

\item (\cite[5.4, 5.5]{SZ})
We call $\xi$-equivariant 
Fano fibration $\pi\colon X\to Y$ is 
{\it K-stable (resp., K-semistable)} 
if and only if 
for any special test configuration, 
Donaldson-Futaki invariant is stable unless it is trivial test configuration 
(resp., non-negative). 
We call $\xi$-equivariant 
Fano fibration $X\to Y$ is 
{\it K-polystable} if and only if it is K-semistable and further 
that the Donaldson-Futaki invariant is $0$ 
only if the special test configuration is of product type i.e., 
$\mathcal{X}\simeq X\times \A^1, \mathcal{Y}\simeq Y\times \A^1$ 
in $T$-equivariant manner. 
\end{enumerate}    
\end{defn}

%cf., also \cite{HanLi} for cpt KRS case 
%or even cpt g-soliton. 

%%%%%%%%%%%%%%%%%%%%%%%%%%%%%%%%%%%%%%%%%%%%%%%%%%%%%%%%%%%%%%%%%%%%%%%%

\section{Examples - Integral computations and 
estimates}

This section discusses explicit computations 
and estimates of the weighted volume $\W(\pi)$ 
in several standard examples. 

\begin{ex}[(Compact) Fano variety case {\cite[Example 4.5]{SZ}}]
When $Y$ is a point, $\W(\pi)$ is an invariant of Fano variety $X$ which is $\frac{(-K_X)^{\cdot n}}{n!}e^{\tilde{\beta}(X)}$ 
with $\tilde{\beta}(X)$ in \cite{HanLi}. 
If $X$ is K-semistable in the original 
sense of Ding-Tian-Donaldson (cf., \cite{Don02}), then 
 $\tilde{\beta}(X)=0$ so that 
 $\W(\pi)=\frac{(-K_X)^{\cdot n}}{n!}$. 
\end{ex}

Next, we review the following simple but important observation by Sun-Zhang, which is 
quite useful for general study of their weighted volume. 

\begin{lem}[Local-global comparison {\cite[cf., Definition 6.5]{SZ}}]\label{lg}
For any Fano fibration $\pi\colon X\to Y\ni p$ and a closed point $q\in \pi^{-1}(p)$, we have 
$$\W({\it id}\colon X\to X\ni q)\ge \W(\pi).$$ 
\end{lem}

Note that the left hand side is essentially purely local 
and the so-called local 
normalized volume 
$\widehat{\vol}(p\in X)$ 
of (kawamata-)log terminal 
singularity $p\in X$ 
discussed in \cite{Li, SS}. More precisely: 

\begin{ex}[Singularities]
If $\pi={\it id}$ i.e., $X\xrightarrow{=}Y\ni p$ 
is the germ of klt singularity, 
as the 
original \cite[Example 4.7]{SZ} explains well, 
\begin{align}
\W(p\in X)=\frac{e^n}{n^n}\widehat{\vol}(p\in X)    
\end{align}
from the definitions 
and the fact $\inf_{A\in \R_{>0}}\frac{e^A}{A^n}=\frac{e^n}{n^n}$. 
In particular, it takes value in $e^n\cdot \overline{\Q}$ by 
\cite[Appendix]{DSII}. 
\footnote{Note that the $2$-step degeneration {\it 
op.cit} does not change the local normalized 
volume, hence one can reduce to the K-polystable 
Fano cone.}
For instance, if $p$ is smooth, then 
$$\W(p\in X)=e^n,$$
which is $7.389\cdots (n=2)$, $20.0855\cdots (n=3)$. 

If $p$ is the ordinary double point, 
$$\W(p\in X)=\dfrac{2((n-1)^n)}{n^n}e^n,$$
which is $3.69\cdots (n=2)$, $11.9025\cdots (n=3)$. 
Spotti-Sun \cite[Conjecture 1.2]{SS} conjectured that this is the second 
largest normalized (local) volume 
(later Liu-Xu \cite{LiuX} proved it in $3$-dimensional 
case 
using the classification theory by Mori and Reid.)
\end{ex}

One can also generalize Lemma \ref{lg} in the same principle: 

\begin{lem}[Generalization of Lemma \ref{lg}]\label{lg:gen}
For a Fano fibration $X\xrightarrow{\pi}Y\ni p$ with affine $Y$, 
and horizontally compactify i.e., take normal quasi-projective 
variety $\overline{Y}\supset Y$ 
(Zariski open), $\overline{X}\xrightarrow{\overline{\pi}}\overline{Y}$ 
so that $\overline{\pi}^{-1}(Y)=X$. 

If a projective morphism $\overline{Y}\xrightarrow{f}Y'\ni p'=f(p)$ 
exists so that $\pi':=f\circ \overline{\pi}\colon \overline{X}\to Y'$ 
is another Fano fibration, one can compare it with $\pi\colon X\to Y\ni p$ 
and 
we have the following inequality: 
\begin{align}
\W(\pi,p)\ge \W(\pi',p').     
\end{align}    
\end{lem}

\begin{proof}
Just recall the definition of weighted volume function. 
For a fixed $v$, both terms $e^{A_X(v)}$ and 
$\int_{\R_{\ge 0}}e^{-x}{\rm DH}(v)dx$ 
only reflects the geometry of the total spaces  but the allowed class of 
$v$ i.e., verticality with respect to $v$ is different. The density function of the 
DH measure of $v$ for 
$\pi'$ is at most that of 
$v$ for 
$\pi$. Further, 
the range of $v$ becomes larger for $\pi'$  compared with $\pi$. Combining these observations, the proof is done. 
\end{proof}

The above inequality can be confirmed in the following basic examples, 
with more explicit values. 

\begin{ex}[$\P^1$-bundle and flat ($\Q$-)Fano fibration]\label{ex:flat}
If $\pi\colon X=\A^1\times \P^1\to Y=\A^1\ni p=0$, 
naturally we can take $E=\pi^{-1}(p)\simeq \P^1$. In this case, 
$\W(\pi)=2e=5.436\cdots$. 

Much more generally, suppose $\pi$ is flat with integral fiber $\pi^{-1}
(p)=F$ with relative dimension $f$ and $\dim(X)=n$ as before. 
Then, since the completion of the generic point of $F$ 
is of the form $\O_p[[z_1,\cdots,z_f]]$, 
a valuation $v$ centered on $p\in Y$ naturally induces a valuation of $K(X)$ 
which we denote as $\pi^*v$ as follows: 
$$(\pi^*v)(\sum_{a_1,\cdots,a_f}c_{a_1,\cdots,a_f}z_1^{a_1}\cdots z_f^{a_f}):=\min\{v(c_{a_1,\cdots,a_f})\mid c_{a_1,\cdots,a_f}\neq 0\},$$
where the minimum exists since ${\rm Im}(v)$ is discrete. 
If we set the valuative ideals (coherent sheaves) as follows: 
for open subsets $U\subset X, V\subset Y$
\begin{align}
(\mathcal{O}_X\supset)\mathcal{J}_{\pi^*v}(x l)
&:=\{f\in \Gamma(\mathcal{O}_U)\mid (\pi^*v)(f)
\ge x l\}, \\ 
(\mathcal{O}_Y\supset)J_{v}(x l)(V)
&:=\{f\in \Gamma(\mathcal{O}_V) \mid v(f)
\ge x l\}, 
\end{align}
where $x, l$ are real number and positive integer respectively. 
Hence, 
$$\pi_*(\mathcal{J}_{\pi^*v}(x l)\cdot L^{\otimes l})=J_{v}(x l)\cdot 
\pi_*L^{\otimes l}$$ so that, combined with 
$(\pi_*L^{\otimes l})_p\simeq \mathcal{O}_{Y,p}^{\oplus h^0(F,L^{\otimes l}|_F)}$ as $\mathcal{O}_{Y,p}$-modules 
and the usual asymptotic Riemann-Roch formula, it easily follows that 
$${\rm DH}_X(\pi^* v_Y)=\binom{n}{f}(-K_X|_F)^{\cdot f}\cdot {\rm DH}_Y(v)$$ 
(cf., \cite{SZ} and Definition \ref{def:DH}). 
Thus, we conclude 
\begin{lem}[Flat ($\Q$-)Fano fibration]\label{lem:bdd.flat}
If $\pi$ is flat with integral $\pi^{-1}(p)$, 
$$\W(\pi)\le e^{n-f}\binom{n}{f}(-K_X|_F)^{\cdot f}\cdot  \widehat{\vol}(p\in Y)$$ 
holds. 
\end{lem}
\noindent
Note that in the right hand side, 
$\W(p\in Y)$ is an invariant of the base while 
$(-K_X|_F)^{\cdot f}$ is nothing but the anti-canonical volume of 
general fibers (as Fano varieties). 

Now, we move on to the case where $\pi$ is {\it not} flat.  
\end{ex}

\begin{ex}[Castelnuovo contraction cf., {\cite[A.7]{CDS}}]
\label{A2blup}
If $X={\rm Bl}_{p}(\A^2)\to Y=\A^2\ni p=0$, 
we confirm (following  \cite[A.7]{CDS}) that 
the weighted volume 
$\W(\pi)$ 
is attained when $v=\frac{\ord_E}{b}$ with 
$E=E_\pi$ the $\pi$-exceptional $(-1)$-curve {\it on} $X$ and 
$b=\frac{1}{\sqrt{2}}$. 
Indeed, for {\it prime} divisor $E$ {\it over} 
$X$ and consider 
$v:=\frac{{\rm ord}_E}{b}$, 
we can estimate/calculate the weighted volume 
as follows. For $b\in \Q_{>0}$ and 
sufficiently divisible $m$, 
\begin{align}
\pi_*\mathcal{O}(-mK_X)&\simeq \pi_*\mathcal{O}(-mE_\pi)
=\mathfrak{m}_{(0,0)}^m=\mathfrak{m}_{Y,p}^m\\ 
\label{30..}\pi_*\mathcal{O}_X\left(-\frac{m}{b}E\right)&\supset 
\mathfrak{m}_{Y,p}^{\frac{m}{b}}\\ 
\pi_*\mathcal{O}_X\left(-mK_X-\frac{m}{b}E\right)&\simeq \pi_*\mathcal{O}_X
\left(-m(E_\pi+\frac{1}{b}E)\right)\\
&=\mathfrak{m}_{Y,p}^{m(1+\frac{1}{b})}
\end{align}
implies that (cf., 
Proposition \ref{DHcal}) 
$${\rm DH}_X(v)\ge \left(1+\frac{1}{b}x\right)e^{-x}dx,$$
with equality holds if and only if $E=E_\pi$. 
Now we set $c:=\frac{1}{b}$. 
\begin{Rem}\label{rem:DH}
Note that the key simple observation 
\eqref{30..} holds for any blow up of 
$Y$ with center supported on $p$ and vertical 
$E$. 
\end{Rem}

Hence 
\begin{align}
\W(v)\ge b\cdot e^{\frac{1}{b}}(1+b)&=e^{c}\cdot \frac{c+1}{c^2}\\ 
&\ge e^{\sqrt{2}}\left(\frac{1+\sqrt{2}}{2}\right)  \hspace {2mm} (c=\sqrt{2} \text{ case})\\
&=\W(\pi)\\
&=4.96\cdots. 
\end{align}

Among smooth $2$-dimensional 
shrinkers whose morphism $\pi$ is birational, 
this weighted volume is the 
second biggest (cf., \cite[Question 6.6]{SZ}), 
though that of $\P^1\times \A^1\to \A^1$ has a larger 
weighted volume 
($=5.4\cdots$ {\text{(cf., Example \ref{ex:flat})}}). 
\end{ex}

The corresponding shrinking soliton metric to the above example \ref{A2blup} 
%and \ref{P3blup} are 
is 
that of \cite[\S 6]{FIK} with $k=1$. 

\begin{ex}[Divisorial contraction to point]
\label{divcont}
More generally, suppose $\pi\colon X\twoheadrightarrow Y$ 
is a divisorial contraction i.e., 
birational projective contraction 
with irreducible exceptional 
divisor $E_\pi$ with log terminal $X, Y$, 
where $-K_{X}$ is $\pi$-ample. 
In other words, $\pi$ is a plt blow up. 
We set the discrepancy $a:=a_Y(E)$ 
which is automatically positive by the 
negativity lemma. 

For here, we further assume the center is 
$0$-dimensional i.e., the closed point $p$. 
Then, similarly as above Example 
\ref{A2blup}, for $c>0$, 

\begin{align}
\W(c\cdot\ord_{E_\pi})
&=\frac{e^c}{c} \int_0^\infty e^{-x} \DH(c\cdot\ord_{E_\pi}) dx \\
&= \frac{e^c}{c}  
\int_0^\infty \frac{1}{c(n-1)!} e^{-x} (a + cx)^{n-1} dx\\ 
&=  \frac{1}{(n-1)!} \frac{e^c}{c^2} \sum_{k=0}^{n-1} \binom{n-1}{k} a^{n-1-k} c^k k!. 
\end{align}

If $Y$ is smooth $n$-dimensional 
and $X$ is a blow up 
at the maximal ideal of closed point $p$, 
then $a=n-1$ so that 
$$\W(c\cdot \ord_{E_\pi})=\frac{1}{(n-1)!} \cdot \frac{e^c}{c^2} \sum_{k=0}^{n-1} \binom{n-1}{k} (n - 1)^{n - 1 - k} k! c^k.$$

If we set a polynomial of $c$ as  
$P(c):=\frac{1}{(n-1)!} 
\sum_{k=0}^{n-1} \binom{n-1}{k} (n-1)^{n - 1 - 
k}k! c^k$, then the critical point (algebraic 
number) should satisfy a polynomial equation 
of degree $n$ with rational coefficients 
$$(c-2)P(c)+c P'(c) = 0,$$
and $\W(c\cdot \ord_{E_\pi})
=\frac{e^c}{c^2} P(c)$. 

Recall that \cite[Chapter 3, especially 3.3, 3.4, 3.5]{Mori} classified 
extremal divisorial 
contraction from smooth $3$-folds, 
which automatically 
includes the (birational) Fano fibrations 
which can be written as resolutions of $3$-dimensional 
log terminal cones. Indeed, 
note that such underlying cones 
have automatically terminal singularities, 
by the negativity lemma (\cite[3.39]{KollarMori}). 
The list is 3.3.1 to 3.3.5 in {\it 
op.cit}, their weighted volume 
can be estimated (well) also by 
the above formula in the same manner. We omit 
the calculation of explicit values here (for 
now). 
For more examples of this type 
and classification results found 
along the later developments; one can refer to 
e.g., \cite[\S 3.2, 3.5, 3.6]{Kawakita}. 
\end{ex}

The following example somewhat mixes Example \ref{ex:flat} ($\P^1$-bundle) 
and Example \ref{A2blup} ($-1$-curve contraction). 

\begin{ex}[{Non-geometric ruled surface \cite[Theorem A]{BCCD}}]\label{ex:BCCD}
If $X={\rm Bl}_{(0,0)}(\P^1\times \A^1)\to Y=\A^1\ni p=0$, 
we set $E_{e}$ as the exceptional divisor of 
${\rm Bl}_{(0,0)}(\P^1\times \A^1)\to\P^1\times \A^1$ and 
the strict transform of $0\times \P^1$ as $F$. Following 
Proposition \ref{DHcal2} again, the density function 
of $\DH$ is $\min\{x/c,1\}$ so that 
we calculate 
\begin{align}
\W(c\cdot \ord_{E_e})
&=\frac{e^c}{c}
\left(\int_0^c e^{-x} \left( \frac{x}{c} + 1 \right) dx + 
\int_c^{\infty} 2e^{-x}dx\right)\\ 
&=\frac{e^c}{c} \left(e^{-c}+(\frac{1}{c}+1)-(\frac{1}{c}+1)e^{-c}\right)\\ 
&=\frac{1}{c}\left(1+(\frac{1}{c}+1)(e^c-1)\right). 
\end{align}
By derivative calculation, the minimizer of the right hand side is 
attained at $c=1.1\cdots$ 
so that 
$\inf_c\W(c\cdot \ord_{E_e})=4.3\cdots$. 
By elementary transform of $X$ along $E$, it follows that 
$$\min_c \W(c\cdot \ord_{F})=\min_c \W(c\cdot \ord_{E_e}).$$

On the other hand, by Lemma \ref{lg}($=$\cite[cf., Definition 6.5]{SZ}), 
$\inf_{v}\W(v)$ where $v$ runs over those whose center is a closed point 
is at most $e^2=7.389\cdots$, as a 
rather weak upper bound. 
More precise calculation is as follows. 
We take any divisor $E$ over $X$ whose center is the 
singular point $q$ of the central fiber of $\pi$. 
Take a normal blow up of $X$ which realizes $E$ as $E\subset X'\xrightarrow{\varphi} X$. 
If we denote the image of $q$ in $X$ as $q'$, as before, we have 
\begin{align}
H^0(X',-m\varphi^*K_X-maE)\supset H^0(X,\mathfrak{m}_{q}^{\lceil ma\rceil}\O(-mK_X))
\end{align}
so that 
the Duistermaat-Heckman measure's 
density function satisfies $\mathcal{R}_{bE}(x)\ge \min\{b^2x,b\}$ 
(the non-differential point of the right hand side is $x=c$). 
Thus, 
\begin{align}
\W(c\ord_E)&\ge \frac{e^{2c}}{c}\left(\int_0^c e^{-x}(bx) +
\int_{c}^{\infty}e^{-x} \right)\\ 
&=\frac{e^{2c}}{c}\left(b\int_0^c x e^{-x} +
e^{-c} \right)\\ 
&=\frac{e^{2c}}{c}\left(b(1-(c+1)e^{-c})) +
e^{-c} \right)\\ 
&=\frac{1}{c}\left(\frac{e^{2c}-(c+1)e^{c}}{c} +
e^{c} \right), 
\end{align}
whose minimum is attained as the exceptional curve ($(-1)$-curve) of the 
blow up of $X$ along $\mathfrak{m}_q$ with $c=0.64\cdots$ so that 
$\W(\pi)=4.1\cdots$. 
\end{ex}

Note that \cite[Theorem A]{BCCD} 
constructed complete K\"ahler-Ricci solitons 
metrics on the above example, 
as a parabolic limit of K\"ahler-Ricci flow 
along the contraction ${\rm Bl}_{(0,0)}(\P^1\times \P^1)\to \P^1$. 
Using that, {\it op.cit} Theorem A 
completed classification of 
$2$-dimensional smooth complete 
K\"ahler-Ricci solitons 
under the 
bounded (scalar/sectional) curvature 
assumption, which is later removed by 
\cite[Theorem 1.2]{LW}. 
Those examples are contained in the above examples, in particular. 

It is a standard exercise to show that 
the list of 
Fano fibrations $X\to Y$ from 
$2$-dimensional smooth surface $X$ are the exact list of {\it loc.cit} 
(even {\it without} K-semistability assumption). 
We also note that Lemma \ref{lg:gen} can be checked between: 
\begin{align}
\W(\P^1\times \P^1\to {\rm pt})&=4 & \\ 
\le \W(\P^1\times \A^1\to \A^1)&=5.4\cdots {\text{(cf., Example \ref{ex:flat})}}, \\
\W({\rm Bl}_{(0,0)}(\P^1\times \A^1)\to \A^1\ni 0)&=4.1\cdots {\text{(cf., Example \ref{ex:BCCD})}}\\
\le \W({\rm Bl}_{(0,0)}(\A^2)\to \A^2\ni (0,0))&=4.9\cdots. 
\end{align}

\begin{ex}[Toric case]
Let us consider $T$-equivariant (klt) Fano fibration 
$T\curvearrowright (X\xrightarrow{\pi} Y\ni p)$, where $X$ is a $T$-toric variety i.e., 
a toric variety with respect to the algebraic torus $T$, 
$Y$ is a $T_Y$-toric variety with surjective homomorphism of 
algebraic tori $T\to T_Y$, so that $\pi$ is $T$-equivariant. 
In this paper, what we mean by 
{\it toric Fano fibration} is such a data 
$T\curvearrowright (X\xrightarrow{\pi} Y\ni p)$ 
(note that $T\to T_Y$ is automatically recoverable from it). 

If $X$ is smooth, and consider 
from symplectic geometric or differential geometric perspectives, 
this fits into the framework of {\it AK-toric (algebraic K\"ahler toric)} 
manifolds introduced and systematically studied by 
C.~Cifarelli \cite{Charlie1, Charlie2}, which 
generalize the Delzant's work \cite{Delzant} to {\it non-compact} ``toric" 
setup. 

In this case, the following is a folklore 
which should be known to experts (at least well-known 
for $\pi={\rm id}$ case cf., e.g., 
\cite{FOW}, \cite[\S 1]{CS19}, \cite[\S 2.5.3]{Od24a} and references 
therein). One would call it a corollary to Conjecture \ref{conj:SZ2step}. 

\begin{prop}
For a $T$-equivariant (klt) Fano fibration 
$T\curvearrowright (X\xrightarrow{\pi} Y\ni p)$, 
suppose $2$-step degeneration conjecture \ref{conj:SZ2step} holds. 
Then, it is K-semistable (in the sense of \cite{SZ}) for some 
$\xi\in N\otimes \R$. 
\end{prop}

\begin{proof}
Take the quasi-monomial valuation $v$ which minimizes the weighted volume 
$\W(v)$ as we assume Conjecture \ref{conj:SZ2step}. 
By its uniqueness, it is $T$-invariant. 
Then following Lemma \ref{asdegen}, we obtain $T$-equivariant 
(isotrivial) degeneration 
$\Pi_\sigma\colon \mathcal{X}_\sigma\to \Y_\sigma\xrightarrow{f_\sigma} 
U_\sigma$ to $T\curvearrowright (X_v\to Y_v)$ and 
this $T\curvearrowright (X_v\to Y_v)$ is $T$-equivariantly 
isomorphic to original $T\curvearrowright (X\to Y)$ 
(cf., \cite[2.32]{Od24a}). Via this isomorphism, this $v$ gives rise to a 
positive vector field $\xi\in N\otimes \R$ (cf., e.g., 
\cite{Od24a}) for $T\curvearrowright Y$, so that the assertion follows 
from Conjecture \ref{conj:SZ2step}. 
\end{proof}

\end{ex}

\begin{ex}(Flipping contraction)
For the flipping contraction case, we leave  
the computations to future, as they  necessarily 
involve divisors {\it above} $X$ or 
non-divisorial valuations. 
Here is a question which Sun-Zhang 
inspires, for the expected 
termination of flips. 

\begin{Ques}
(See \cite[Last paragraph of \S 6.3]{SZ})
If there would be infinite sequence of flips 
$\{X_i\xrightarrow{\pi_i} Y_i\leftarrow X_i^+=X_{i+1}\} \quad (i=1,2,\cdots)$ 
in the (fixed) minimal model program, 
in particular, what can we say about its growth 
of the sequence 
$\{\W(\pi_i)\}_i$? 

It is obviously {\it not}  
necessarily monotonically increasing in 
general, but 
for any $i$, is there some big enough $i'$ 
such that $\W(\pi_i)<\W(\pi_{i'})$ for 
instance? 
\end{Ques}
Note that the latter would contradicts if 
$\{\W(\pi)\mid \dim(X)=n\}$ satisfies 
ACC and would lead to flip termination. 
\end{ex}

Now we come back to the general situation of Fano fibrations. 
Motivated by the above examples calculations and 
our formulae such as 
Proposition \ref{DHcal2} \eqref{divcase4} (also Prop \ref{DHcal}), we 
conclude this section by asking the arithmetic nature of 
weighted volumes. 

\begin{Ques}[cf., Kontsevich-Zagier \cite{KZ}]
Is weighted volume $\W(\pi)$ of Fano fibration germ 
$\pi\colon X\to Y\ni p$ always an exponential period in the sense of 
Kontsevich-Zagier (\cite[\S 4.3]{KZ}) 
or some variant (cf., e.g., 
\cite{CHH})? 
\end{Ques}

%%%%%%%%%%%%%%%%%%%%%%%%%%%%%%%%%%%%%%%%%%%%%%%%%%%%%%%%%%%%%%%%%%%%%%%%

\section{More general theoretic aspects}

The topics discussed in this section are of general theoretic nature, 
which center around the moduli theory of Fano fibrations, 
as well as relation with the theory of Fano cones, among others. 

\subsection{Partial reduction to log terminal  
cone case}

In this subsection, we observe that for $T$-equivariant Fano fibration, 
one can associate Fano cone (log terminal cone), which do not lose the 
information. 

\begin{prop}\label{Ffka}

\begin{enumerate}
\item 
\label{rc1} 
For a Fano fibration 
$\pi\colon X\to Y$, 
its (relative) cone 
$$C_Y(L):=\Spec (\oplus_{l\in \Z_{\ge 0}}{H^0(X,L^{\otimes l})})$$ is log terminal and 
$\pi\colon C_Y(L)\to Y$ is a relative affine K-trivial fibration. 
\item 
\label{rc2}
For a $T$-equivariant Fano fibration 
$T \curvearrowright (X\xrightarrow{\pi}Y)$, 
its (relative) cone 
$C_Y(L)=\Spec (\oplus_{l\in \Z_{\ge 0}}{H^0(X,L^{\otimes l})})$ 
is a log terminal 
cone (Fano cone) with respect to a good $T\times \G_m$-action. 
\end{enumerate}

\end{prop}

\begin{proof}

Note that $\pi_* L^{\otimes l}$ is a coherent sheaf on $Y$ 
with $T$-action (linearization) of each $l$ hence they correspond to 
finitely generated $\Gamma(\O_Y)$-module $\Gamma(\pi_* L^{\otimes l})
=H^0(X,L^{\otimes l})$. Moreover, they form a finite type graded $\Gamma(\O_Y)$-algebra. 
Consider the relative cone $C_Y(L)={\rm Spec}_Y \oplus_{l\ge 0} \pi_* L^{\otimes l}
={\rm Spec}H^0(L^{\otimes l})$. 

\eqref{rc1} 
Firstly, we show \eqref{rc1}  i.e., that $C_Y(L)$ is log terminal, so that it is a log terminal cone (Fano 
cone) 
with respect to the $T\times \mathbb{G}_m$-action. 
Consider the blow up of the vertex section $Z:=V(\oplus_{l>0} R_l)\simeq Y$, 
then you obtain $p\colon {\rm BC}(-rK_X)={\rm Spec}_{X}\oplus_{m\ge 0}L^{\otimes l}
={\rm Bl}_{Z}(C_Y(L))\to C_Y(L)$. 
Here, ${\rm Bl}_Z$ denotes the blow up along $Z$. 
We denote the exceptional divisor (with coefficient $1$) as 
$E\simeq Y$. We have 
$$K_{{\rm BC}(-rK_X)}=p^* K_{C_Y(L)}+(r-1)E$$ as in \cite[3.13, 3.14(4)]{Kol13}. 
On the other hand, since $({\rm Bl}_{Z}(C_Y(L)),E)$ is \'etale (or analytically) locally 
isomorphic to $X\times \mathbb{A}^1$ (resp., $(X \times \mathbb{A}^1, X\times 0)$)  
outside $E$ (resp., near $E$), it is purely log terminal. Hence, 
$C_Y(L)$ is klt.

Now we show \eqref{rc2}. 
For $l\ge 0$, we consider the $T$-eigendecomposition of $\Gamma(Y,\pi_*L^{\otimes l})=\oplus_{l,
\vec{m}}R_{l,\vec{m}}$ and put $\Gamma_l:=\{\vec{m}\in M\mid  R_{l,\vec{m}} \neq 0\}$. 
Since $\Gamma(\pi_*L^{\otimes l})$ is a finitely generated $\Gamma(\O_Y)$-module, 
there exists $\vec{m}_0\in M$ and a strictly convex rational polyhedral cone $\mathcal{C}\subset M\times \mathbb{R}$ such that 
$\Gamma_l\subset l\vec{m}_0+l\mathcal{C}$. Hence, 
$\cup_{l\ge 0}\Gamma_l$ is also inside a strictly convex rational polyhedral cone in $M\times \mathbb{R}\ni 
(\vec{m},l)$. 
Thus, the $T\times \mathbb{G}_m$-action on $C_Y(L)$ is a good action. 

\end{proof}

\begin{ex}

Let us consider the classical Example 
\ref{A2blup} i.e., 
when $X\to Y$ is a blow up of the origin at $\mathbb{A}_{z_1,z_2}^2$ with the exceptional divisor $e$, 
take $r=\frac{1}{2}$ so that $L=\mathcal{O}(1)=\mathcal{O}(-e)$. 
Then, $C_Y(L)$ is a quadratic cone $(z_1Z_2-Z_1z_2=0)\subset \mathbb{A}_{z_1,z_2}^2
\times \mathbb{A}_{Z_1,Z_2}^2=\mathbb{A}^4$ i.e., the (absolute) cone over $\mathbb{P}^1\times 
\mathbb{P}^1$ with respect to $\mathcal{O}(1,1)$. 
If we do consider the obvious higher 
dimensional generalization i.e., 
the blow up of the origin at $\mathbb{A}_{z_1,\cdots,z_n}^n$ with $L=\mathcal{O}(1)$ (cf., Example 
\ref{divcont}), 
then $C_Y(L)$ is the absolute cone over a 
Fano manifold which is an irreducible 
component 
cut by quadratic equations in $\mathbb{P}^{2n-
1}$. 

\end{ex}

Note that the Fano cone $T\times \G_m\curvearrowright 
C_Y(L)$ recovers the Fano fibration 
$X\to Y$; because $Y={\rm Spec}
(\Gamma(\O_{C_Y(L)})^{\G_m})$ 
and $R_l (l>0)$ can be also 
recovered as the eigen-subspace for the 
$\G_m$-action. From this perspective, 
one can naturally ask the following 
interesting question: 

\begin{Ques}[Reduction to cone]
For a $T$-equivariant 
Fano fibration 
$T \curvearrowright (X={\rm Proj}_{Y}(\oplus_{l\in \Z_{\ge 0}}{R_l})\to Y)$, 
we take the relative cone (Fano cone) 
${\rm Spec}_Y\oplus_{l\in \Z_{\ge 0}}{R_l}\to Y$ as Proposition \ref{Ffka} \eqref{rc2}. 

Would there be any relation between the  K-(semi)stability notion  
and other (possibly ``weighted") 
stability notions 
of the relative cone $C_Y(L)$, regarded as  
Fano cones? 
See e.g., Mabuchi-Nakagawa conjecture 
\cite{MN, ALL} in the same spirit. 
\end{Ques}

%At least the minimizing valuations often have 
%different centers. 

%%%%%%%%%%%%%

\subsection{Review of the $2$-step degeneration theory}\label{sec:2stepreview}

After the original $2$-step degeneration theory \cite{DSII, CSW} and later more algebro-geometric implementation by \cite{Li} for the former, 
\cite[Conjecture 6.4 (also 6.8)]{SZ} conjectures the following, which we 
briefly recall for completeness. 

\begin{Setup}
For any (real) valuation $v$ of the function field $K(X)$ of $X$, whose center $q$ lies inside $\pi^{-1}(p)$, 
suppose that both graded ring 
${\rm gr}_v(\oplus_{l\ge 0} H^0(L^{\otimes l}))$ and $\gr_v(\O_{Y,p})$ are of finite 
type. Then, we consider polarized fibration 
$X_v:=\Proj_{Y_v}({\rm gr}_v(\oplus_{l\ge 0} H^0(L^{\otimes l}))\to Y_v:=\Spec(\gr_v(\O_{Y,p}))$. 

Let $M$ be the groupification of the value group of $v$ (called the holomorphic spectrum \cite{DSII}), and let $N$ be its dual lattice. Set $T := N \otimes \G_m$. 
The natural groupification function 
$v\colon M\to \R$ is identified with a vector $\xi\in N\otimes \R$. 
Note that $T$ acts equivariantly on $X_v\to Y_v$. 
\end{Setup}

\begin{conj}[{\cite[Conjecture 6.4]{SZ}}]\label{conj:SZ2step}
For any Fano fibration over an affine pointed variety 
$X\xrightarrow{\pi}Y\ni p$, 
there is a unique quasi-monomial (hence, real valued) 
valuation $v$ of $K(X)$ whose center is 
supported inside $\pi^{-1}(p)$ 
and 
minimizes the weighted volume $\W(-)$ i.e., 
achieves $\W(\pi)$. 
The associated 
$T\curvearrowright (X_v\to Y_v)$ is K-semistable (Fano fibration) with repect to $\xi$, which comes from Lemma \ref{asdegen}. 
\end{conj}

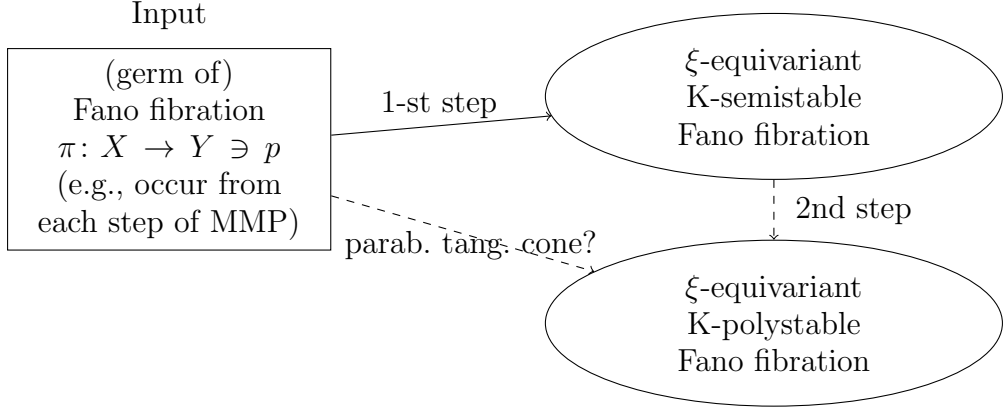
\begin{figure}\label{Fig0}
\begin{tikzpicture}
\usetikzlibrary{positioning}
    \tikzset{Start/.style={rectangle,  draw,  text centered, text width=4cm, align=center}};
    \tikzset{Process/.style={ellipse,  draw,  text centered, text width=4cm, align=center}};
     \tikzset{Process2/.style={ellipse,  draw,  text centered, text width=4cm, align=center}};  
     \tikzset{Remain/.style={circle,  draw,  text centered, text width=0.5cm, align=center}};  
     \tikzset{Remain2/.style={circle,  draw,  text centered, text width=0.2cm, align=center}}; 
 \node[Start](a)at (-1,0.8){(germ of) Fano fibration\\  $\pi\colon X\to Y\ni p$\\ 
 (e.g., occur from each step of MMP)}; 
 \node[above=3pt of a] {Input};
 \node[Process] (b) at (7,1.5){$\xi$-equivariant K-semistable Fano fibration}; 
\node[Process2](c)at (7,-1.5){$\xi$-equivariant K-polystable Fano fibration};
    \draw[->]  (a) --(b)node[xshift=-126pt,yshift=-3pt]{$1$-st step};
       \draw[->, dashed]  (a) --(c)node[above,xshift=-114pt,yshift=+20pt]{parab.\ tang.\ cone?};
       \draw[->, dashed]  (b) --(c)node[above,xshift=+30pt,yshift=+33pt]{2nd step};   
\end{tikzpicture}

\caption{$2$-step degenerations (\cite[\S 6.2, \S 6.4]{SZ})}\label{2step}
\end{figure}

Sun-Zhang also gives a conjectural description of the minimizing valuation 
$v$ 
via the K\"ahler-Ricci flow (\cite[\S 6.4, (6.5)]{SZ}, cf., also 
analogous \cite[(3.4)]{CSW}, \cite[2.27, 2.28]{Od24c}). 

%%%%%%%%%%%%%

\subsection{Compact moduli spaces of K-polystable equivariant Fano fibrations} %(or complete gradient shrinking weak K\"ahler-Ricci solitons)}

\subsubsection*{Preparation}
To proceed to discussions related to moduli theory, 
we first discuss two (related) preparatory materials: 
\begin{enumerate}
\item \label{bbss} boundedness of K-semistable $T$-{\it equivariant} Fano fibrations,  
\item (bigger) parameter space of $T$-{\it equivariant} Fano fibrations, including 
\eqref{bbss}. 
\end{enumerate}
The key to the former is 
the following conjecture by Sun-Zhang. 

\begin{conj}[Boundedness {\cite[\S 6.3, after Conj. 6.7]{SZ}}]\label{conj:bdd}
For any real positive number $c>0$, 
the set of isomorphism class of 
$T$-equivariant ($\Q$-)Fano fibrations 
$T\curvearrowright (X\xrightarrow{\pi}Y)$ 
whose weighted volumes $\W(\pi)$ are at least 
$c$, are bounded i.e., parametrized inside a finite type (quasi-compact) $\k$-scheme. 
\end{conj}

It is well-known that boundedness issue 
is a necessary preparatory 
step for many moduli construction 
which is essentially 
independent from other steps; 
note that actual construction of 
moduli space (as certain nontrivial {\it quotient} of good/semistable locus) 
usually involves independent 
discussions, which normally involve finer analysis such as stabilities. 
\footnote{recall the original 
construction of $M_g$ in \cite{Mum65}.}
For the case of relative dimension $0$, the above conjecture is solved by 
\cite{XZ}, after many substantial progresses such as 
\begin{itemize}
\item 
\cite{HLQ, LMS} (including $2$-dimensional case), 
\item \cite{LMS, ZhuangII}
(including $3$-dimensional case) and 
\item \cite{Jiang} (quasi-regular case). 
\end{itemize}

\begin{Rem}[Boundedness]
Given the recent deep boundedness 
results of 
C.~Birkar \cite{Birkar.before, BAB, Birkar.sing, Birkar.bdd} 
and Birkar-Chen \cite{BC} on boundedness of Fano fibrations and their 
singularities, 
the above conjecture \ref{conj:bdd} 
seems to follow once one somewhat develops along their line. 
Indeed, firstly, 
from the local-global comparison of weighted volume 
(Lemma \ref{lg}$=$\cite[after Def. 6.5]{SZ}) combined with \cite{XZ}, it follows that the (log terminal) singularities which appear on the total space $X$ are bounded. 
This already leads to some non-trivial 
boundedness via \cite{BAB, BC, Birkar.sing, 
XZ} i.e., boundedness of fibers and the bases as 
follows. 

\begin{prop}[Boundedness of fibers and base]
For fixed positive integer $n$, non-negative integer $f$ and $\epsilon>0$, set 
\begin{align}
\mathcal{S}_{\epsilon,f,n}&:=\{\text{$T$-equivariant $\Q$-Fano fibration } T\curvearrowright(X\xrightarrow{\pi}Y)\mid \dim(X)=n,{\rm rdim}(\pi)=f,\W(\pi)>\epsilon\},\\ 
\mathcal{S}'_{\epsilon,f,n}&:=\{[T\curvearrowright(X\xrightarrow{\pi}Y)]\in 
\mathcal{S}_{\epsilon,f,n} 
\mid \dim(X)=n,{\rm rdim}(\pi)=f,\W(\pi)>\epsilon, Y:\text{$\Q$-factorial}\}. 
\end{align}
Here ${\rm rdim}(\pi)$ means the 
relative dimension of $\pi$ i.e., $\dim(X)-\dim(Y)$. 
For the latter, 
recall that $\Q$-factoriality of $Y$
holds when $X$ is $\Q$-factorial and $\pi$ is elementary 
extremal contraction 
(cf., e.g., 
\cite[3.18]{KollarMori}). Moreover, obviously 
$\mathcal{S}'_{\epsilon,f,n}=\mathcal{S}'_{\epsilon,f,n}$ if the base dimension 
$n-f$ is at most $2$. 
These sets $\mathcal{S}_{\epsilon,f,n}$ 
and $\mathcal{S}'_{\epsilon,f,n}$ 
satisfy the following boundedness type results: 
    \begin{enumerate}
\item (fibers' boundedness) \label{36i} general fibers of $\mathcal{S}_{\epsilon,f,n}$, which are $\Q$-Fano varieties are bounded 
\item (singularity of base) \label{36ii} there exists $\delta>0$ such that 
the base $Y$ for $\pi\in \mathcal{S}'_{\delta,f,n}$ 
are all $\delta$-lc. 
\item (bases' boundedness) the base $Y$ for $\pi\in \mathcal{S}'_{\delta,f,n}$ are bounded. 
\end{enumerate}
\end{prop}
\begin{proof}
Firstly, consider the item \eqref{36i}. 
The subset of $\mathcal{S}_{\epsilon,f,n}$ 
with smooth $X$, one can use 
smoothness of the generic fiber (generic 
smoothness) 
and apply \cite{KMM} to prove 
\eqref{36i}. For general case, we apply 
essentially the 
same idea but with more technicalities: 
by the local-global comparison lemma \ref{lg} 
(cf., \cite[after Def. 6.5]{SZ}) 
and the finite degree formula of 
the local normalized volume \cite{XZ0} 
(cf., also \cite{Li, SS} etc), applied to (local) index $1$ cover, 
$\Q$-Cartier indices of the total space  
$X$ are uniformly bounded above by a constant. Hence, in particular, 
there is some uniform $\epsilon>0$ 
such that for any $\pi\in 
\mathcal{S}_{\epsilon,f,n}$, $X$ is 
$\epsilon$-log terminal. Combined with 
the simple generic adjunction (cf., \cite[5.17]{KollarMori}), the general 
fibers 
are also uniformly 
$\epsilon$-log terminal for uniform 
$\epsilon>0$. Given this arguments, 
the first item \eqref{36i} now 
follows from the famous result of 
Birkar \cite{BAB} (Borisov-Alexeev-Borisov 
conjecture). 

The second item follows from \cite[1.3]{BC} (cf., also \cite{Birkar.before}, 
\cite[1.2]{Birkar.sing}), 
combined with the canonical bundle formula 
\cite{FM}. 
The last item 
then follows from \eqref{36ii} 
(or Lemma \ref{lem:bdd.flat} for flat cases) 
combined with \cite{XZ}. 
\end{proof}
The remaining subtle problem seems to lie in 
the following: 
\begin{Ques}[variation or weight control]
For $[T\curvearrowright (X\xrightarrow{\pi}Y)]\in \mathcal{S}_{\epsilon,d,n}$, 
give a uniform upper bound of 
the weights of 
$T\curvearrowright H^0(Y,-
lK_Y)$ 
(and $T\curvearrowright H^0(X,-
lK_X)$)  
for fixed $l\gg 0$. 
\end{Ques}
The author expects this is related 
to stability of the base. 
\end{Rem}

\subsubsection*{Lower semicontinuity} 
The (expected) lower semicontinuity of 
$\W(\pi)$ with respect to variation of 
the family $\pi$ also seems to approachable 
by the method of using Birkar's 
bounded complements (\cite[6.4]{BLXZ})  
combined with the relative versions 
developed in \cite[see e.g., 
Theorem 1.7]{Birkar.bdd}. We do not discuss 
further details in this paper. 

\begin{Setup}(Preparing parameter space)\label{set:par2}
Note that for a $T$-equivariant Fano fibration 
$T \curvearrowright (X\xrightarrow{\pi}Y)$, 
its (relative) cone 
$C_Y(L)=\Spec (\oplus_{l\in \Z_{\ge 0}}{H^0(X,L^{\otimes l})})$ with its 
good $T\times \G_m$-action obviously recovers 
$T \curvearrowright (X\xrightarrow{\pi}Y)$, as it is so for 
{\it family} of $T$-equivariant Fano fibrations as well. 
Motivated by this fact, we consider 
the $N\times \Z_{\ge 0}$-graded 
ring $\oplus_{l\in \Z_{\ge 0}}{H^0(X,L^{\otimes l})}$ and 
its homogeneous generators. 
Suppose that $s$ of the generators have weights $0$ 
for the $\G_m$-action i.e., base direction, 
and the remained $u+1$ of them have weights $\vec{0}$ 
for the $T$-action i.e., fiber direction. 

Then, consider a 
multi-graded Hilbert scheme (\cite{HS, AZ}) which parameterizes the 
corresponding embedding into $\A^{s+u+1}$ and denote it by ${\rm MH}$. 
This can be taken as a finite type scheme over $\k$ as we assume 
Conjecture \ref{conj:bdd}. 
By \cite[1.2]{HS}, it is a projective scheme. 
Using \cite[Cor 24]{hull} to stratify ${\rm MH}$, to obtain a 
quasi-projective (${\rm MH}-$)scheme $H$ which parametrizes 
$T\times \G_m$-equivariant ($\Q$-Gorenstein family of) $\Q$-Fano cones 
including $C_Y(L)$. By e.g., \cite[Lemma 3.1]{Kol13}, it can be seen as 
parameter space of $T$-equivariant Fano fibrations (whose total space is 
admissible (resp., $T$-faithfully flat) in the sense of \cite{HS} (resp., \cite{Od24a})). 

Note that the centralizer of $T\times \G_m$ in ${\rm GL}(\A^{s+u+1})$ is 
reductive and denote by $G$. Then, $G\curvearrowright H$ preserves the 
isomorphic classes of $T$-equivariant fibrations and its quotient stack 
$[H/G]$ 
can be regarded as their moduli stack. 

Note that the $2$-step degeneration conjecture 
\cite[Conjecture 6.4]{SZ} (as natural generalization of \cite{DSII, CSW}) 
expects existence of 
degeneration of $T\curvearrowright (X\xrightarrow{\pi} Y)$ 
to K-semistable $T$-equivariant Fano fibration $T\curvearrowright (Z\to W)$. 
We expect that if we fix the multi-Hilbert function of 
$T\curvearrowright (X,L)\to Y$ and 
take many enough homogeneous generators of the $N\times \Z_{\ge 0}$-graded ring 
$\oplus_l H^0(X,L^{\otimes l})$ to embed $C_Y(L)$ to $\A^{s+u+1}$ 
and consider the above $H$, 
then all associated $T\curvearrowright (Z\xrightarrow{\pi_Z} W)$ 
and degeneration to them 
of Lemma \ref{asdegen} type along an affine variety $U_\sigma$ 
are realized inside $H$ i.e., 
by some morphism $U_\sigma\xrightarrow{m_{\pi_Z}}H$ for each $\pi_Z$. 
We call such $H$, a {\it big enough} parameter scheme of 
$T$-equivariant Fano fibrations. 
\end{Setup}

For moduli construction and its properness, we use the above 
parameter scheme in Setup \ref{set:par2} and expect the following 
stratification structure, after \cite{AHLH, Od24b}: 

\begin{conj}[{Higher $\Theta$-stratification cf., \cite[\S 3]{Od24b}}]\label{hts}
Considering the class of $T$-equivariant 
Fano fibrations $\pi\colon X\to Y$, 
fix multi-Hilbert function 
of $T\curvearrowright Y$ and the 
Hilbert polynomial of $-K_X$ at $\pi$-fibers. 

Then, there is a parameter scheme 
$H\curvearrowleft G$ of 
$T$-equivariant Fano fibrations ($G$-action 
preserves the fibrations isomorphism class), 
which is 
big enough in the sense as above, and 
the weighted volume $\W(-)$ is lower semicontinuous and induces a 
{\it higher $\Theta$-stratification} with finite strata 
$\{\mathcal{Z}_c:=\{\W(-)=c\}\}_c$
on 
$[H/G]$, in the sense of \cite[Definition 3.17]{Od24a} (extending \cite{AHLH}),  
which encode the generalized test configurations of Lemma \ref{asdegen} 
familywise in the form $\mathcal{Z}_c\times [U_\sigma/T]\to \mathcal{Z}_c$. 
\end{conj}

We closely follow the construction method of K-moduli space of 
Calabi-Yau cone in \cite{Od24a} and 
generalize it to that of 
K-polystable $T$-equivariant Fano fibrations. For that, 
there are several steps and the best proof of properness would require 
affirmative confirmation of the above conjecture: 

\begin{prop}
Consider the locus of $H$ where $\xi$-equivariant Fano fibrations are 
K-semistable and denote as $H^{\rm kss}$. If $[H^{\rm kss}/G]$ 
admits a good moduli space (resp., which is separated) in the sense of 
\cite{Alper}, 
Conjecture \ref{hts} implies that it is universally closed 
(resp., proper). 
\end{prop}

\begin{proof}
The proof follows exactly the same method as \cite[\S 3.5]{Od24a}, using 
the higher $\Theta$-semistable reduction theorem 
\cite[Theorem 1.1 (Theorem 3.8 for details)]{Od24b} 
(cf., also \cite[\S 6]{AHLH}, \cite[\S 7]{BHLINK}). 
\end{proof}

Expanding the definition in 
\cite[Definition 3.17]{Od24b} 
which generalizes \cite[\S 6]{AHLH}, 
note that the higher $\Theta$-stratification  conjecture \ref{hts} means the existence of 
{\it family-wise} version of the 
predicted $2$-step degeneration (Conjecture \ref{conj:SZ2step}=\cite[Conjecture 6.4]{SZ}). 
Following the degeneration theoretic perspective 
after Lemma \ref{asdegen}, 
it is to make it simultaneous i.e., over 
higher dimensional base of the form 
$S\times U_\sigma$ with some variety $S$. 
By valuative criterion of properness, one can reduce to the case when 
$S$ is a smooth (pointed) 
curve and then (again) 
essentially a ubiquitous ``finite generation" 
type problem in birational geometry after \cite{BCHM} 
(as in \cite{ABHLX}). Indeed, once the finite generation property is 
confirmed, its spectrum automatically satisfies certain K-semistability 
as analogous to the CM minimization phenomenon (cf., \cite{Od20, Hat}). 

\subsection{Bubbling Fano fibrations}

In this subsection, we present an algebro-geometric construction - which we refer to as ``bubbling" - of certain Fano fibrations, under a technical conjectural assumption 
on some stack structure 
(Conjecture \ref{hts2}). Roughly speaking,  starting from a given degeneration of Fano fibrations, we construct “asymptotically conical” Fano fibrations in a relatively canonical way, as a kind of rescaled limit (see Theorem \ref{Mthm.bble}).

We expect that this construction can be recoverable 
in a differential geometric manner, hence the name, 
by using good K\"ahler metrics family 
and consider the differential geometric 
bubbling i.e., non-trivial rescaled limit. 
See \cite{Od24c} and differential geometric 
references therein for the special case when $\pi={\rm id}$. 
Nevertheless, we mainly focus on purely algebro-geometric side. 
See Figure 2 for the outlook of the algorithm. 

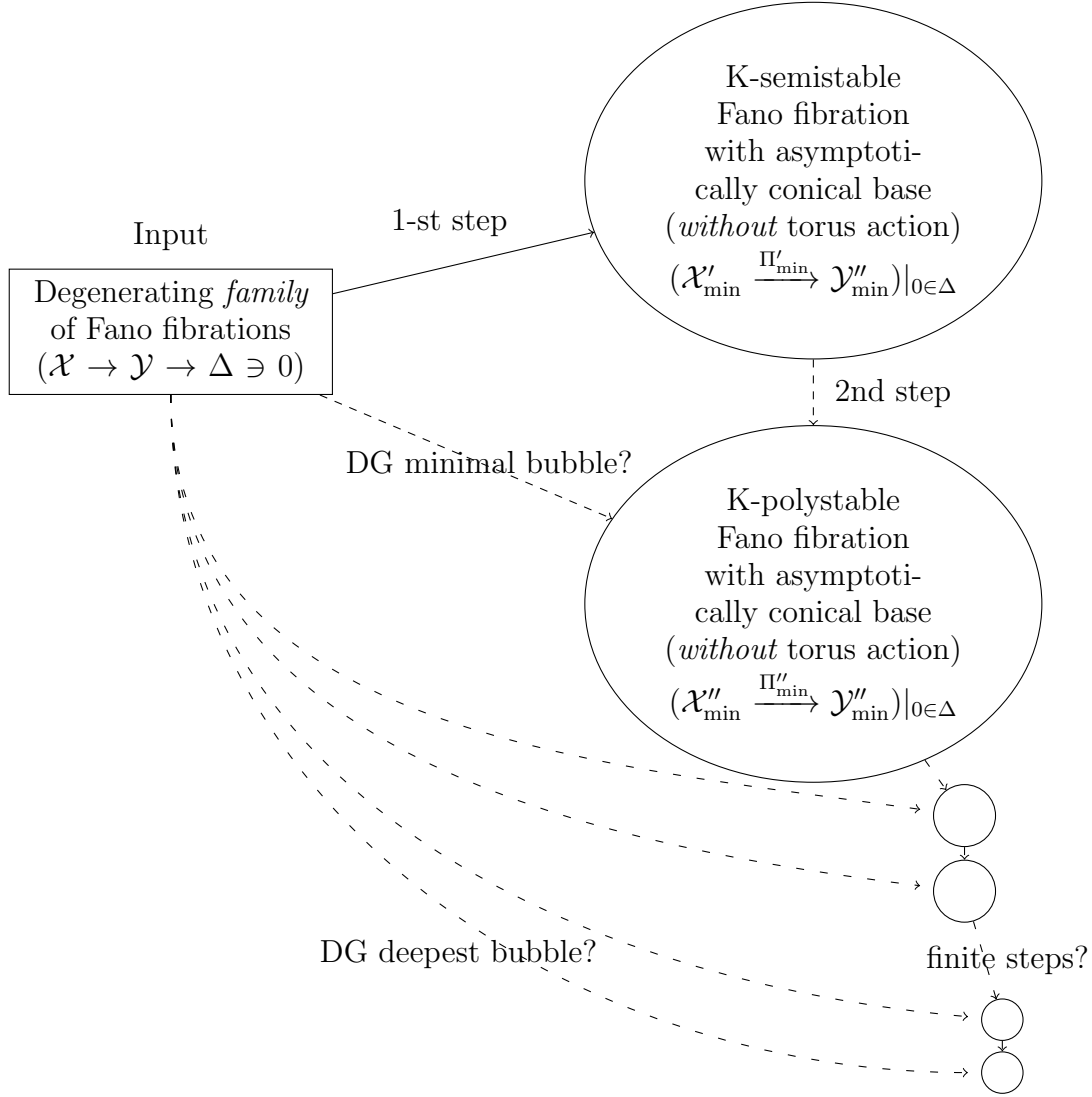
\begin{figure}\label{Fig1}
\begin{tikzpicture}
\usetikzlibrary{positioning}
    \tikzset{Start/.style={rectangle,  draw,  text centered, text width=4cm, align=center}};
    \tikzset{Process/.style={ellipse,  draw,  text centered, text width=4cm, align=center}};
     \tikzset{Process2/.style={ellipse,  draw,  text centered, text width=4cm, align=center}};  
     \tikzset{Remain/.style={circle,  draw,  text centered, text width=0.5cm, align=center}};  
     \tikzset{Remain2/.style={circle,  draw,  text centered, text width=0.2cm, align=center}}; 
 \node[Start](a)at (-1,0.8){Degenerating {\it family} of Fano fibrations\\ $(\X\to \Y\to \Delta\ni 0)$};
 \node[above=3pt of a] {Input}; 
 \node[Process](b)at (7.5,2.8){K-semistable Fano fibration \\ with asymptotically conical base \\ ({\it without} torus action) \\ 
 $(\mathcal{X}_{\rm min}'\xrightarrow{\Pi'_{\rm min}}\mathcal{Y}''_{\rm min})|_{0\in \Delta}$}; 
\node[Process2](c)at (7.5,-2.8){K-polystable Fano fibration \\ with asymptotically conical base \\ ({\it without} torus action) \\ $(\mathcal{X}_{\rm min}''\xrightarrow{\Pi''_{\rm min}}\mathcal{Y}''_{\rm min})|_{0\in \Delta}$};
 \node[Remain](d)at (9.5,-5.6){}; 
  \node[Remain](e)at (9.5,-6.6){}; 
   \node[Remain2](f)at (10,-8.3){}; 
    \node[Remain2](g)at (10,-9){}; 
    \draw[->]  (a) --(b)node[xshift=-137pt,yshift=-16pt]{$1$-st step};
       \draw[->, dashed]  (a) --(c)node[above,xshift=-122pt,yshift=+45pt]{DG minimal bubble?};
       \draw[->, dashed]  (b) --(c)node[above,xshift=+30pt,yshift=+70pt]{2nd step};   
       \draw[->, loosely dashed]  (c) --(d)node[]{};
        \draw[->, loosely dashed]  (d) --(e)node[]{};
         \draw[->, loosely dashed]  (e) --(f)node[above,xshift=+2pt,yshift=+13pt]{finite steps?};
                  \draw[->, loosely dashed]  (f) --(g)node[]{};
\draw[->, loosely dashed, out=270, in=172, shorten >=5pt]  (a) to (d);
\draw[->, loosely dashed, out=270, in=173, shorten >=5pt]  (a) to (e);
\draw[->, loosely dashed, out=270, in=175, shorten >=5pt]  (a) to (f);
\draw[->, loosely dashed, out=270, in=180, shorten >=5pt] (a) to node[above,xshift=14pt,yshift=-34pt]{DG deepest bubble?} (g);
\end{tikzpicture}

\caption{Bubbling Fano fibrations cf., Thm \ref{Mthm.bble} \& Prop. \ref{bbletermination} \\ 
    \qquad (we allow finite base changes of $\Delta$). DG stands for 
    ``Differential Geometric"} \label{AE.Od}
\end{figure}

To establish such construction (Theorem \ref{Mthm.bble}), we 
first make a technical slight generalization of the previous 
Conjecture \ref{hts}. 
As in Conjecture \ref{hts}, 
techniques by proving finite generations, by essentially reducing to 
\cite{BCHM}, should apply. 

\begin{conj}[General existence of higher $\Theta$-stratification]\label{hts2}
Take any (finite type) algebraic $\k$-stack $\mathcal{M}^o$ 
which underlies 
a $\Q$-Gorenstein faithfully family of 
$\Q$-Fano fibrations, i.e., 
$\tilde{\Pi}_{\X}\colon \tilde{\mathcal{X}}\xrightarrow{\tilde{\Pi}}\tilde{\mathcal{Y}}\xrightarrow
{\tilde{\Pi}_{\tilde{\mathcal{Y}}}}\mathcal{M}^o$ 
where $\tilde{\Pi}_{\tilde{\mathcal{Y}}}$ is a faithfully flat 
affine morphism, $\tilde{\Pi}$ is a ($\Q$-)Fano fibration, 
together with a section $\sigma$ of $\tilde{\Pi}_{\tilde{\mathcal{Y}}}$ 
i.e., $\tilde{\Pi}_{\tilde{\mathcal{Y}}}\circ \sigma={\rm id}$. 

Then, there is a monomorphism 
$\mathcal{M}^o \to \mathcal{M}$ where $\mathcal{M}$ is 
another (still finite type) 
\footnote{being parallel to the big enoughness of the previous 
subsection} quotient 
algebraic $\k$-stack 
$\mathcal{M}$ which underlies 
a $\Q$-Gorenstein family of 
$\Q$-Fano fibrations which extends $\tilde{\Pi}_{\X}$ 
and a 
higher $\Theta$-stratification on $\mathcal{M}$ 
(in the sense of \cite[\S 3]{Od24b})
defined by the weighted volume function 
which is lower semicontinuous and constructible 
i.e., finite strata of the forms 
$\{\mathcal{Z}_c:=\{\W(-)=c\}\}_{c\in 
\R}$, and it encodes the generalized test configurations of Lemma \ref{asdegen} 
familywise in the form $\mathcal{Z}_c\times [U_\sigma/T]\to \mathcal{Z}_c$. 
\end{conj}

Note that the above 
formulation implicitly contains several 
smaller conjectures; for instance 
i.e., the lower semicontinuity of 
the weighted volume, the boundedness and hence some ACC type nature of the 
set of weighted volumes. 

The main difficulty of the conjecture 
is to show the properness of the evaluation morphism 
${\rm ev}_{(1,\cdots,1)} \colon \mathcal{Z}^{+} \to \mathcal{M}$, 
where $\mathcal{Z}^{+}\subset {\rm Map}(\Theta_\sigma,\mathcal{M})$ 
is a union of connected components, 
in the notation of \cite[\S 3, around Definition 3.3]{Od24b}. 
To prove it by the valuative criterion of 
properness (universally closedness, to be 
precise), 
eventually this should be approachable by proving finite generation 
problem again, by technically but eventually reducing to \cite{BCHM}. 

For the case of relative dimension $0$ Fano fibrations i.e., 
family of klt singularities germs, there are related discussions 
to this conjecture in 
\cite{Od24a, ZChen, Od24c}. Supposing the above, we extend the bubbling 
construction of \cite{Od24c} as follows. We also prepare the following 
notion: 

\begin{defn}[Graded negative valuation cf., \cite{SZ1}]\label{grnv}
For a projective family $\pi\colon X\to Y$ over an affine algebraic 
$\k$-scheme $Y$, 
with a relative ample line bundle 
$L$ on $X$, 
we set $R_l:=\Gamma(Y,\pi_*L^{\otimes l})=\Gamma(X,L^{\otimes l})$. 
A {\it graded negative valuation} 
\footnote{the term ``negative valuation" comes from the earlier work of S.Sun and 
J.Zhang (cf., \cite[\S 6]{SZ1}). 
Note that the minus $-d$ satisfies the axiom of valuations indeed. 
One could also call it simply ``degree function" (or generalized 
degree) since the classical 
degree of multi-variable polynomials is a typical example of $d$.}
of $\oplus_{l\in \Z_{\ge 0}}R_l$ is a function 
$d\colon (\oplus_{l\in \Z_{\ge 0}}R_l)\setminus \{0\}\to \R_{\ge 0}$ 
satisfying the following properties: 
\begin{enumerate}
    \item \label{nv1} $d(\sum_{l}x_l)=\max\{d(x_l)\mid x_l\neq 0\}$ for any $x=\sum_l x_l$ 
    where $x_l$ denotes the component of $R_l$, 
    or equivalently 
    $d$ is $\G_m$-invariant, 
    \item $d(xy)=d(x)+d(y)$, 
    \item $d(x+y)\le \max\{d(x),d(y)\}$, 
\end{enumerate}
for any $x,y \in (\oplus_{l\in \Z_{\ge 0}}R_l)$. 
For each $d$, we can define a $\R_{\ge 0}$-graded ring 
$$\gr_d(\oplus_l R_l):=\oplus_{a\in \R_{\ge 0}}\{x\in \oplus R_l\mid d(x)\le 
a\}/\{x\in \oplus R_l\mid d(x)< a\}.$$ 
As in the previous subsection, 
let $M$ denote the groupification of the image semigroup 
${\rm Im}(d) \subset \R_{\ge 0}$, set 
its dual lattice $N$ and algebraic torus $T:=N\otimes \G_m$, 
it naturally has a $T$-action. We set $M_{\ge 0}:=M\cap \R_{\ge 0}$. 
From the first condition \eqref{nv1} 
above, we can decompose this as a $\Z_{\ge 0}\times M_{\ge 0}$-graded ring 
\begin{align}\label{gr1}
\gr_d(\oplus_l R_l)=\oplus_{l\in \Z_{\ge 
0}}(\oplus_{a\in \R_{\ge 0}}\{x\in R_l\mid d(x)\le a\}/\{x\in R_l\mid d(x)< a\}).
\end{align}
\end{defn}

We denote $$Y_d:={\rm Spec}(\gr_d(\oplus_l R_l)).$$ 
Compatibly, we can consider 
a $M_{\ge 0}$-graded ring 
\begin{align}\label{gr2}
\gr_d(\Gamma(\mathcal{O}_Y)):=\oplus_{a\in \R_{\ge 0}}\{x\in \Gamma(\mathcal{O}_Y)\mid d(x)\le 
a\}/\{x\in \Gamma(\mathcal{O}_Y)\mid d(x)< a\}.
\end{align}
Then, $\gr_d(\oplus_l R_l)$ is a $\Z_{\ge 0}\times \R_{\ge 0}$-graded 
$\gr_d(\Gamma(\mathcal{O}_Y))$-algebra. If \eqref{gr1} is of finite type, 
one can consider 
relative spectra 
$${\rm Spec}_{\gr_d(\Gamma(\mathcal{O}_Y))}(\gr_d(\oplus_l R_l))$$ 
(resp., $$X_d:={\rm Proj}_{\gr_d(\Gamma(\mathcal{O}_Y))}(\gr_d(\oplus_l R_l)))$$
as affine (resp., polarized projective) $Y_d$-variety. 
In that case, there is a generalized test configuration over an affine toric variety $U_\sigma$ for a rational polyhedral cone $\sigma$ 
in $N\otimes \R$ of 
$X\to Y$ degenerating to $X_d\to Y_d$ exactly as in 
\cite[Example 2.18]{Od24b} (compare Lemma \ref{asdegen}). 

Here is the simple generalization of the 
notion by S.~Sun \cite{Song}, which 
corresponds to the case $\pi={\rm id}$. 

\begin{defn}[Fano fibration with 
asymptotically conical base]
An {\it Fano fibration 
with asymptotically conical base} means a 
Fano fibration germ $(X\xrightarrow{\pi}Y\ni 
p)$, together with a graded negative valuation $d$ in the above sense of 
Definition \ref{grnv}, such that $T\curvearrowright(X_d\to Y_d)$ 
is a $T$-equivariant (klt) Fano fibration in the sense of 
Definition \ref{def:eqff}. 

Generalizing the terminology of 
\cite[\S 5]{Song}, we call such 
Fano fibration 
with asymptotically conical base 
$[(X\xrightarrow{\pi}Y\ni 
p),d]$ 
is {\it K-polystable} (resp., {\it K-stable, K-semistable}) 
if $T\curvearrowright(X_d\to Y_d)$ 
is so as a $T$-equivariant (klt) Fano fibration. 
\end{defn}

Now we follow 
\cite[Theorem 2.4 (or cf., 1.1 for a quick overview)]{Od24c} closely 
to give a construction of certain (K-semi/polystable) 
Fano fibration with asymptotically conical base, which we call algebro-geometric minimal bubblings. 
{\it Loc.cit} treats the case when $\pi$ is trivial. 

Among other results in this paper, 
the following is the main one in this subsection. 

\begin{Thm}[Minimal bubbling Fano fibrations]\label{Mthm.bble}
Suppose the above Conjecture \ref{hts2} holds. 
Let $C$ be a pointed smooth curve with base closed point $0 \in C$, and 
consider an arbitrary {\bf $\Q$-Gorenstein family of (klt) Fano fibrations} over $C\ni 0$ as 
$$\Pi_\X\colon \mathcal{X}\xrightarrow{\Pi} \mathcal{Y}\xrightarrow{\Pi_{\Y}} C\ni 0$$ with the section 
$\sigma\colon C\to \mathcal{Y}$ ($\Pi_\Y \circ \sigma={\rm id}$) 
i.e., $-K_\X$ is $\Q$-Cartier and 
$\Pi$-ample, such that $\X_s:=\Pi_\X^{-1}(s)\to \Y_s:=\Pi_{\Y}^{-1}(s)\ni \sigma(s)$ 
for closed point $s\in S$ 
has same 
weighted volumes for $s\neq 0$ while it becomes {\bf strictly 
smaller} for $s=0$. 

Then, 
after a finite base change $R\colon C'\to C$ of 
$0\in C$, 
there is a modification along the preimage over $s=0$ 
to have another family of Fano fibrations 
\begin{align}
\Pi_{\X'_{\rm min}}\colon \mathcal{X}_{\rm min}'\xrightarrow{\Pi'_{\rm min}} 
\mathcal{Y}'_{\rm min}\xrightarrow{\Pi_{\Y'_{\rm min}}} C',\\ 
\text{(resp., }\Pi_{\X''_{\rm min}}\colon \mathcal{X}_{\rm min}''\xrightarrow{\Pi''_{\rm min}} 
\mathcal{Y}''_{\rm min}\xrightarrow{\Pi_{\Y''_{\rm min}}} C'\text{)}, 
\end{align}
which satisfy the following properties: 
\begin{enumerate}
    \item 
    The induced family over $C' \setminus R^{-1}(0)$ 
    $$(R\circ \Pi_{\X'_{\rm min}})^{-1}(C\setminus 0)
    \to (R\circ \Pi_{\Y'_{\rm min}})^{-1}(C\setminus 0)\to C'\setminus R^{-1}(0)$$ 
    (resp. $(R\circ \Pi_{\X''_{\rm min}})^{-1}(C\setminus 0)
    \to (R\circ \Pi_{\Y''_{\rm min}})^{-1}(C\setminus 0)\to C'\setminus R^{-1}(0)$) 
    is (both) isomorphic to the fiber product of 
    $\X\to \Y\to C$ with $C'\setminus R^{-1}(0)\to C\setminus 0\hookrightarrow C$, so that in particular 
    we have the unique section $\sigma'$ (resp., $\sigma''$) 
    compatible with $\sigma$ (by the valuative criterion of properness). 
    \item (Increase of weighted volume) For any $s'\in C'$ with $R(s')=0$, 
    the weighted volume of $\Pi_{\X'_{\rm min}}^{-1}(s')\to \Pi_{\Y'_{\rm min}}^{-1}(s')\ni \sigma'(s')$ 
    (and $\Pi_{\X''_{\rm min}}^{-1}(s')\to \Pi_{\Y''_{\rm min}}^{-1}(s')\ni \sigma''(s')$)
    is {\bf strictly larger} than 
    $\X_0\to \Y_0\ni \sigma(0)$. 
    \item (K-semistability resp., K-polystability) There is a graded negative valuation $d$ of  
    $\Pi_{\X'_{\rm min}}$ (resp.,  of $\Pi_{\X''_{\rm min}}$) 
    with which they are {\bf K-semistable resp., K-polystable 
    (klt) Fano fibration with asymptotically conical base.} 
    
    That is, 
    \eqref{gr1} and \eqref{gr2} are both finite type $\k$-algebra and 
    $X_d\to Y_d$ is K-semistable (resp., K-polystable) 
    in the sense of \cite{SZ}. 
\end{enumerate}
\end{Thm}

\begin{proof}
We follow the construction (proof) of 
\cite[Theorem 2.4]{Od24c}
closely, which corresponds to the case when the morphism $\Pi$ is identity. 
Given the recent general higher $\Theta$-semistable reduction theorem 
\cite[Theorem 3.8]{Od24b} (cf., also \cite[\S 7]{BHLINK}), 
the main discussions here is to set up a certain parameter space (stack). 

After the $2$-step degeneration 
conjecture \ref{conj:SZ2step}($=$\cite[Conjecture 6.4]{SZ}), which we 
assume, we set the K-semistable degeneration of $X\xrightarrow{\pi}Y$ 
as $N\otimes \G_m=T\curvearrowright (X_v\xrightarrow{\pi_v}Y_v)$. 
Here, by the valuation $v$, direction $\xi\in N\otimes \R$ is determined. 
(Original \cite[Conjecture 6.4]{SZ} denotes them as 
$Z\to W$.) We take $n(\gg 1)$ homogeneous generators of 
$\Gamma(W,\O_W)\curvearrowleft 
T$ and embed $W$ into $\A^n$ accordingly. 
Following the proof of \cite[2.4]{Od24c}, 
we take the defining equations of $Y_v\subset \A^n$ 
as $f_1,\cdots,f_N$ and set $d_i:=\deg_\xi(f_i)$. 
Then, as in {\it loc.cit}, consider affine $\xi$-negative 
deformations of $Y_v$ as 
$\{V(\{f_i+h_i\}_{i=1,\cdots,N})\mid 
0<{\rm deg}_{\xi}(h_i)<d_i\}_{h_i}$, 
apply the flattening stratification to its natural parameter space 
(affine space for the coefficients of $h_i$), so that 
we obtain a $T$-equivariant 
affine flat deformations over an 
affine $\k$-scheme ${\rm Def}^{-}(Y_v)\subset \A^m$ for $m\gg 0$. 
We denote the obtained deformation as $\tilde{\mathcal{Y}}\xrightarrow{\pi_{\tilde{\mathcal{Y}}}}{\rm Def}^-(Y_v)$. 
After this preparation, we consider the relative multi-Hilbert 
scheme over ${\rm Def}^{-}(Y_v)$ as \cite{AZ} and \cite[1.1,1.2]{HS} 
which we denote as ${\rm MH}(\pi_{\tilde{\mathcal{Y}}})\to {\rm Def}^-(Y_v)$. 
This ${\rm MH}(\pi_{\tilde{\mathcal{Y}}})$ parametrizes 
polarized projective fibrations over affine deformations of 
$Y_v$. Then we take the universal hull 
\cite[1.2]{hull} to it to obtain 
$M(\pi_{\tilde{\mathcal{Y}}})(\to {\rm MH}(\pi_{\tilde{\mathcal{Y}}}))$ 
which parametrizes (klt) $\Q$-Gorenstein families of $X$ with 
their fibering over affine deformations of $Y_v$. 
Consider the stack $\mathcal{M}^o:=[M(\pi_{\tilde{\mathcal{Y}}})/T]$ and 
apply Conjecture \ref{hts2} to obtain a quotient $\k$-stack $\mathcal{M}$ 
with higher $\Theta$-stratification $\{\mathcal{Z}_c\}_c$. 
Then, now we can apply \cite[Theorem 3.8]{Od24b} as in the proof of 
\cite[2.4]{Od24c}. 
In particular, we obtain $\xi$-negative 
$\Q$-Gorenstein degeneration family along $V(\tau')\simeq \A^1$ in the proof of 
\cite[Theorem 3.8]{Od24b} 
\footnote{To be precise we apply 
\cite[Theorem 3.8]{Od24b} 
with DVR $R:=\O_{C,0}$. 
The outcome is, for a Galois covering 
$C'\to C$, ${\rm Gal}(C'/C)$-invariant 
$C'$-valued point of $\mathcal{M}$. }
of a Fano fibration, 
which we denote by 
$\Pi_{\X'_{\rm min}}\colon \mathcal{X}_{\rm min}'\xrightarrow{\Pi'_{\rm min}} 
\mathcal{Y}'_{\rm min}\xrightarrow{\Pi_{\Y'_{\rm min}}} C'$, 
with the central fibers $X_v\to Y_v$. 
Since $X_v$ and $Y_v$ are both irreducible, 
this degeneration corresponds to graded negative valuation $d$ 
of rank $1$. 

We next construct 
$\Pi_{\X''_{\rm min}}\colon \mathcal{X}_{\rm min}''\xrightarrow{\Pi''_{\rm min}} 
\mathcal{Y}''_{\rm min}\xrightarrow{\Pi_{\Y''_{\rm min}}} C'$, 
following the proof of \cite[2.4]{Od24c} again. It can be done in a 
completely parallel manner by combining the method of 
\cite[Theorem 1.3, \S 3]{LWX} 
and 
the construction of $\X''_{\rm min}$ in 
the proof of \cite[2.4]{Od24c}. 
We complete the proof. \end{proof}

\begin{defn}[Minimal bubblings]
We call the above 
$\Pi_{\X'_{\rm min}}^{-1}(s')\to \Pi_{\Y'_{\rm min}}^{-1}(s')\ni \sigma'(s')$ 
    (resp., $\Pi_{\X''_{\rm min}}^{-1}(s')\to \Pi_{\Y''_{\rm min}}^{-1}(s')\ni \sigma''(s')$) for $R(s')=0$
{\it minimal K-semistable bubbling Fano fibrations} 
(resp., {\it minimal K-polystable bubbling Fano fibrations}) 
after \cite{Song, dBS, Od24c}. 
Note that 
by the Galois (${\rm Gal}(C'/C)$-)invariance of 
the above bubbling 
construction, these do not depend on 
$s'\in R^{-1}(0)$. 
\end{defn}

\begin{prop}\label{bbletermination}
    (Finite time termination) 
We continue to use the setup of Theorem \ref{Mthm.bble} (still assuming 
Conjecture \ref{hts2}).     If we repeat the replacement $\Pi_\X$ by $\Pi_{\X''_{\rm min}}$ 
    finite times, then it stops in the sense that 
    the weighted volume functions of the fibers over the base curve become constant.    
\end{prop}    

\begin{proof}
This follows immediately because we only have finite strata in the 
higher $\Theta$-stratification as assumed in Conjecture \ref{hts2}. 
\end{proof}

\begin{Ques}
Clarify the differential geometric meaning of the above algebro-geometrically constructed 
K-polystable minimal bubbling Fano fibrations $\Pi_{\X''_{\rm min}}$, 
as a certain bubbling (rescaled limits as metric spaces). 
\end{Ques}

Note that these construction depend on a priori non-canonical construction of parameter space (stack) $\mathcal{M}$, which is morally 
regarded as ``finite dimensional slice" of 
infinite dimensional deformation space of $T\curvearrowright(X_v\to Y_v)$, as well as 
on its special test configuration 
to the polystable limit. This a priori 
non-canonicity may parallel to the 
non-uniqueness of the metrics on $\pi^{-1}(t) 
(t\neq 0)$. 
Recall that in the case of $\pi={\rm id}$, 
under some conditions, 
\cite{Song, dBS, Od24c} proves that this 
$\Pi_{\X''_{\rm min}}$ 
can be understood as bubbling limit of K\"ahler-Einstein metrics 
under certain situations. 
In that setup, \cite{dBS, Od24a} 
observes canonicity of the bubblings in some 
examples. 

\begin{Rem}
Recall that \cite[Cor 2.9]{Od24c}, as a consequence of 
{\it op.cit} Theorem 2.4, is a variant 
\footnote{it is weaker than the usual resolution of singularities of Hironaka type but our 
process is more canonical with connections to differential geometry. 
Also, it is more global construction than classical Zariski's local 
uniformization} 
of resolution (alteration) of log terminal singularities. 
Similarly, 
above Theorem \ref{Mthm.bble} and Proposition \ref{bbletermination} 
can be morally regarded as a {\it family/fibration version of} 
alteration of log terminal singularities. 
\end{Rem}

%%%%%%%%%%%%%%%%%%%%%%%%%%%%%%%%%%%%%%%%%%%%%%%%%%%%%%%%%%%%%%%%%%%%%%%%

\begin{ack}
The author thanks S.Sun and 
J.Zhang for many helpful discussions on their original \cite{SZ} 
which led to this notes. 
The author also thanks S.Mori and R.Conlon for 
helpful comments on the draft. We also 
sincerely thank the referee for careful reading 
and helpful suggestions. 
This paper is written for a special volume of PAMQ in 
honor of Prof. Caucher Birkar. As this paper and the main proof  
in \cite{SZ} 
also show, his 
various works in birational geometry have played a crucial role 
in the recent advances on K-stability 
and K\"ahler geometry. 

During the work, the author was partially 
supported by Grant-in-Aid for Scientific Research (B) 23H01069, 21H00973, 
Grant-in-Aid for Scientific Research (A) 21H04429, 
20H00112, 
25H00586, 
and Fund for the Promotion of Joint International Research (Fostering Joint International Research) 23KK0249, 
and the National Science 
Foundation under Grant No. DMS-1928930, while he was in
residence at the Simons Laufer Mathematical Sciences Institute
(formerly MSRI) in Berkeley, California, during the Fall 2024. 
\end{ack}

%%%%%%%%%%%%%%%%%%%%%%%%%%%%%%%%%%%%%%%%%%%%%%%%%%%%%%%%%%%%%%%%%%%%%%%%

\vspace{5mm} \footnotesize \noindent
Contact: {\tt yodaka@math.kyoto-u.ac.jp} \\
Department of Mathematics, Kyoto University, Kyoto 606-8285. %JAPAN \\

\end{document}